\documentclass{article}
\usepackage[utf8]{inputenc}

\usepackage{graphicx}
\usepackage{comment}
\usepackage{enumitem}
\usepackage{tikz}
\usepackage{array}
\usepackage{pgfplots}
\usetikzlibrary{pgfplots.fillbetween}
\usepackage{mathtools}

\usepackage{graphicx}

\usepackage{amsmath}
\usepackage{amsfonts}
\usepackage{amssymb}
\usepackage{amsthm}
\usepackage{hyperref}

\usetikzlibrary{shapes,backgrounds}
\usetikzlibrary{shapes.misc, positioning}
\usepackage{pgfplots}
\usepgfplotslibrary{fillbetween}
\usetikzlibrary{intersections}
\pgfdeclarelayer{ft}
\pgfsetlayers{main,ft}
\usetikzlibrary{positioning, fit}

\usetikzlibrary {positioning}
\tikzset{main node/.style={circle,fill=blue!40,draw,minimum size=2cm,inner sep=0pt},}

\usepackage{epsfig}
\usepackage{graphicx}

\usepackage{xcolor}

\usepackage{pst-node}
\usepackage{tikz-cd}
\pgfplotsset{compat=1.17}

\newtheorem{proposition}{Proposition}
\newtheorem{corollary}{Corollary}

\newtheorem{lemma}{Lemma}

\title{The classification of preordered spaces in terms of monotones: complexity and optimization} 

\author{Pedro Hack, Daniel A. Braun, Sebastian Gottwald}
\date{ }

\begin{document}

\maketitle

\begin{abstract}
The study of complexity and optimization in decision theory involves both partial and complete characterizations of preferences over decision spaces in terms of real-valued monotones. With this motivation, and
following the recent introduction of new classes of monotones, like injective monotones or strict monotone multi-utilities, we present the classification of preordered spaces in terms of both the existence and cardinality of real-valued monotones and the cardinality of the quotient space. In particular, we take advantage of a characterization of real-valued monotones in terms of separating families of increasing sets in order to obtain a more complete classification consisting of classes that are strictly different from each other. As a result, we gain new insight into both complexity and optimization, and clarify their interplay in preordered spaces.
\end{abstract}

\section{Introduction}

The question of how well a preorder relation can be captured through real-valued functions is an ongoing research topic since the introduction of utility functions in the early days of mathematical economics. The key observation is that sometimes preferences can not only be measured locally to decide between two elements, but there might be a global real-valued preference function that fully captures the corresponding order relation. That is, in certain situations, one can not only choose a preferable item between any two items in a given set of options, but one can find a single function, or a family of functions, defined on the decision space whose function values quantify the preference relation, so that one can compare function values to decide about the order relation of the corresponding arguments. Since the existence of such functions only depends on the properties of the corresponding preorder, this idea can naturally be applied in many domains of science. In particular, instead of considering preference relations and utility functions on decision spaces, many systems of interest can be thought of as sets of possible states endowed with an order relation encapsulating the intrinsic tendency of the system to transition from one state to another. The fields where these ideas are relevant include \emph{thermodynamics} \cite{lieb1999physics,giles2016mathematical}, \emph{general relativity} \cite{bombelli1987space,minguzzi2010time}, \emph{quantum physics} \cite{nielsen1999conditions,brandao2015second} and \emph{economics} \cite{debreu1954representation,ok2002utility}, among others.  

The basic property of these real-valued functions $f$ is that they have to be \emph{monotones} with respect to the corresponding preorder $\preceq$, that is, $x\preceq y$ implies $f(x)\leq f(y)$. There are mainly three types of monotones that appear in this context: strict monotones \cite{alcantud2016richter,peleg1970utility,richter1966revealed}, injective monotones \cite{hack2022representing}, and utility functions \cite{debreu1954representation,debreu1964continuity}. In particular, these different types of monotones are used to classify preordered spaces mostly in two different ways: either by whether a given type of monotone exists, or, by whether there exists a family of such monotones, known as a \emph{multi-utility}, that characterizes the preorder completely \cite{evren2011multi,alcantud2013representations,alcantud2016richter,hack2022representing,bosi2012continuous,bosi2016continuous}. Equivalently, these spaces can be classified according to either the existence of optimization principles with certain characteristics or the complexity of the preorder, that is, the amount and type of multi-utilities that exist for them. Even though the cardinality of such representing families plays an important role, so far mostly the two cases of \emph{countable} multi-utilities and multi-utilities consisting of a \emph{single} element, that is, utilities, have been considered.


 
Moreover, several connections between both types of classifications have been pointed out in the literature \cite{alcantud2016richter,hack2022representing,alcantud2013representations,bosi2018upper}, but certain gaps in these connections have prevented the presentation of
 a general classification of preordered spaces through real-valued monotones. One of the aims of this contribution is to reduce this gap, achieving, thus, a more complete classification (see Figure \ref{fig:classification}) and, hence, a better understanding of both complexity and optimization, including how they are related, in preordered spaces.
In particular, we take advantage of a characterization of real-valued monotones in terms of families of increasing sets \cite{alcantud2013representations,hack2022representing} that allows to distinguish more classes of preordered spaces than before, both in terms of the cardinality of the multi-utilities and the cardinality of the quotient space of the preorder. Importantly, by providing the corresponding counter examples, we show that certain classes of preordered spaces are in fact strictly contained in each other, which, to our knowledge, was not known before.

\begin{figure}[!tb]
\centering
\begin{tikzpicture}[scale=0.75, every node/.style={transform shape}]
\node[ draw, fill=blue!5 ,text height = 16cm,minimum     width=16cm, label={[anchor=south,above=1.5mm]270: \textbf{Multi-utility}}]  (main) {};
\node[ draw, fill=blue!15, text height =12cm, minimum width = 12cm,xshift=1cm,yshift=1cm,label={[anchor=south,below=1.5mm]90: \textbf{Strict monotone}}] at (main.center)  (semi) {};
\node[ draw, fill=blue!15, text height =12.5cm, minimum width = 12cm,xshift=-1cm,yshift=-0.5cm,label={[anchor=south,above=1.5mm]270: \textbf{Multi-utility with cardinality $\mathfrak{c}$}}] at (main.center)  (semi) {};
\node[ draw, text height =12cm, minimum width = 12cm,xshift=1cm,yshift=1cm] at (main.center)  (semi) {};
\node[ draw,fill=blue!30, text height = 11cm, minimum width = 10.5cm,xshift=-1cm,yshift=-0.3cm,label={[anchor=south,above=1.5mm]270:$|X/\mathord{\sim}| \leq \mathfrak{c}$}] at (main.center) (active) {};
\node[ draw, text height =12cm, minimum width = 12cm,xshift=1cm,yshift=1cm] at (main.center)  (semi) {};
\node[ draw,fill=blue!30,yshift= 0.375cm,text height = 10.75cm, minimum width = 10cm,label={[anchor=south,above=1.5mm]275:\textbf{Strict monotone multi-utility with cardinality $\mathfrak{c}$}}] at (main.center) (active) {};
\node[ draw, text height = 11cm, minimum width = 10.5cm,xshift=-1cm,yshift=-0.3cm] at (main.center)  (semi) {};
\node[ draw,fill=blue!50, text height = 8cm, minimum width = 7cm,label={[anchor=south,above=1.5mm]270:\textbf{Injective monotone}}] at (main.center) (active) {};
\node[ draw,fill=blue!70, text height = 6cm, minimum width = 6cm,label={[anchor=south,above=1.5mm]270:\textbf{Countable multi-utility}}] at (main.center) (non) {};
\node[ draw,fill=blue!80, text height = 4cm, minimum width = 4cm,yshift=-0.1cm,xshift=0.6cm,label={[anchor=south,above=0.1mm]270: \textbf{Countable $|X/\mathord{\sim}|$}}] at (main.center) (non) {};
\node[ draw,fill=blue!80, text height = 4cm, minimum width = 4cm,yshift=0.7cm,xshift=-0.5cm,label={[anchor=north,below=1.5mm]90:\textbf{Finite multi-utility}}] at (main.center) (non) {};
\node[ draw,fill=blue!10, text height = 2cm, minimum width = 2.8cm,xshift=-0.5cm,label={[anchor=north,below=1cm]90:\textbf{Utility function}}] at (main.center) (non) {};
\node[ draw, text height = 4cm, minimum width = 4cm,yshift=-0.1cm,xshift=0.6cm] at (main.center) (non) {};
\node[ draw, text height = 4cm, minimum width = 4cm,yshift=-0.1cm,xshift=0.6cm] at (main.center) (non) {};
\end{tikzpicture}
\caption{Classification of preordered spaces according to the existence of various real-valued monotones. A distinction between our contributions here and previously known results can be found in the discussion (Section \ref{discussion}). Moreover, our contributions can be visualized in Figure \ref{fig:contributions}.}
\label{fig:classification}
\end{figure}
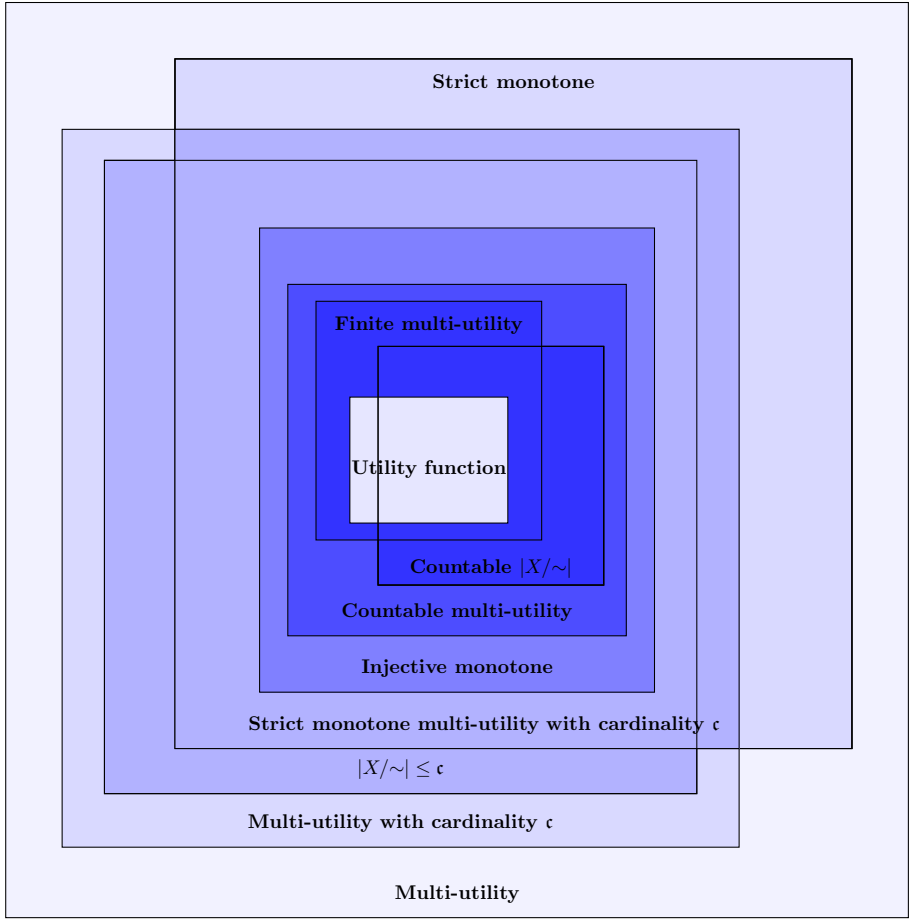

\section{Real-valued monotones and their role in complexity and optimization}

In this section, we introduce the classes of real-valued monotones that are relevant throughout this work and relate them to both  optimization and complexity. Before entering the general picture, we motivate them through an example with several applications, namely, the uncertainty preorder $\preceq_U$ \cite{hack2022representing}.

\begin{figure}[!tb]
\centering
\begin{tikzpicture}[scale=0.75, every node/.style={transform shape}]
\node[ draw, fill=blue!5 ,text height = 16cm,minimum     width=16cm, label={[anchor=south,above=1.5mm]270: \textbf{Multi-utility}}]  (main) {};
\node[ draw, fill=blue!15, text height =12cm, minimum width = 12cm,xshift=1cm,yshift=1cm,label={[anchor=south,below=1.5mm]90: \textbf{Strict monotone}}] at (main.center)  (semi) {};
\node[ draw, fill=blue!15, text height =12.5cm, minimum width = 12cm,xshift=-1cm,yshift=-0.5cm,label={[anchor=south,above=1.5mm]270: \textbf{Multi-utility with cardinality $\mathfrak{c}$}}] at (main.center)  (semi) {};
\node[ draw, text height =12cm, minimum width = 12cm,xshift=1cm,yshift=1cm] at (main.center)  (semi) {};
\node[ draw,fill=blue!30, text height = 11cm, minimum width = 10.5cm,xshift=-1cm,yshift=-0.3cm,label={[anchor=south,above=1.5mm]270:$|X/\mathord{\sim}| \leq \mathfrak{c}$}] at (main.center) (active) {};
\node[ draw, text height =12cm, minimum width = 12cm,xshift=1cm,yshift=1cm] at (main.center)  (semi) {};
\node[ draw,fill=blue!30,yshift= 0.375cm,text height = 10.75cm, minimum width = 10cm,label={[anchor=south,above=1.5mm]275:\textbf{Strict monotone multi-utility with cardinality $\mathfrak{c}$}}] at (main.center) (active) {};
\node[ draw, text height = 11cm, minimum width = 10.5cm,xshift=-1cm,yshift=-0.3cm] at (main.center)  (semi) {};
\node[ draw,fill=blue!50, text height = 8cm, minimum width = 7cm,label={[anchor=south,above=1.5mm]270:\textbf{Injective monotone}}] at (main.center) (active) {};
\node[ draw,fill=blue!70, text height = 6cm, minimum width = 6cm,label={[anchor=south,above=1.5mm]270:\textbf{Countable multi-utility}}] at (main.center) (non) {};
\node[ draw,fill=blue!80, text height = 4cm, minimum width = 4cm,yshift=-0.1cm,xshift=0.6cm,label={[anchor=south,above=0.1mm]270: \textbf{Countable $|X/\mathord{\sim}|$}}] at (main.center) (non) {};
\node[ draw,fill=blue!80, text height = 4cm, minimum width = 4cm,yshift=0.7cm,xshift=-0.5cm,label={[anchor=north,below=1.5mm]90:\textbf{Finite multi-utility}}] at (main.center) (non) {};
\node[ draw,fill=blue!10, text height = 2cm, minimum width = 2.8cm,xshift=-0.5cm,label={[anchor=north,below=1cm]90:\textbf{Utility function}}] at (main.center) (non) {};
\node[ draw, text height = 4cm, minimum width = 4cm,yshift=-0.1cm,xshift=0.6cm] at (main.center) (non) {};
\node[ draw, text height = 4cm, minimum width = 4cm,yshift=-0.1cm,xshift=0.6cm] at (main.center) (non) {};

\filldraw (2.25,1.5) circle (3pt) node[above=0.3mm] {\textbf{A}};
\filldraw (-2.25,-2) circle (3pt) node[above=0.3mm] {\textbf{B}};
\filldraw (3.8,-2) circle (3pt) node[above=0.3mm] {\textbf{C}};
\filldraw (4.65,-2) circle (3pt) node[above=0.3mm] {\textbf{D}};
\filldraw (-5.5,-2) circle (3pt) node[above=0.3mm] {\textbf{E}};
\filldraw (4.65,-6.5) circle (3pt) node[above=0.3mm] {\textbf{F}};
\filldraw (6,1.5) circle (3pt) node[above=0.3mm] {\textbf{G}};
\filldraw (6,-6.5) circle (3pt) node[above=0.3mm] {\textbf{H}};
\end{tikzpicture}
\caption{Contributions of this work to the classification of preordered spaces. We reproduce here Figure \ref{fig:classification}, incorporating a point for each preorder we have introduced that has allowed us to distinguish between classes. In particular, \textbf{A} stands for the preorder in Proposition \ref{no finite mu},  \textbf{B} for the one in Proposition \ref{improve}, \textbf{C} for the preorder in both Proposition \ref{strict mono no injective} and Corollary \ref{R multi-ut no inj mono}, \textbf{D} for that of Proposition \ref{mulit-ut and cardi} (ii) taking $I=\mathbb{R}$, \textbf{E} for the one in Corollary \ref{cardin no strict mono}, \textbf{F} for the preorder in Corollary \ref{last coro} taking $I=\mathbb{R}$, \textbf{G} for the space in Proposition \ref{strict mono no R multi} taking $I=\mathbb{R}$ and, lastly, \textbf{H} for the one in Proposition \ref{no strict mono no R multi} taking $I=\mathbb{R}$.}
\label{fig:contributions}
\end{figure}
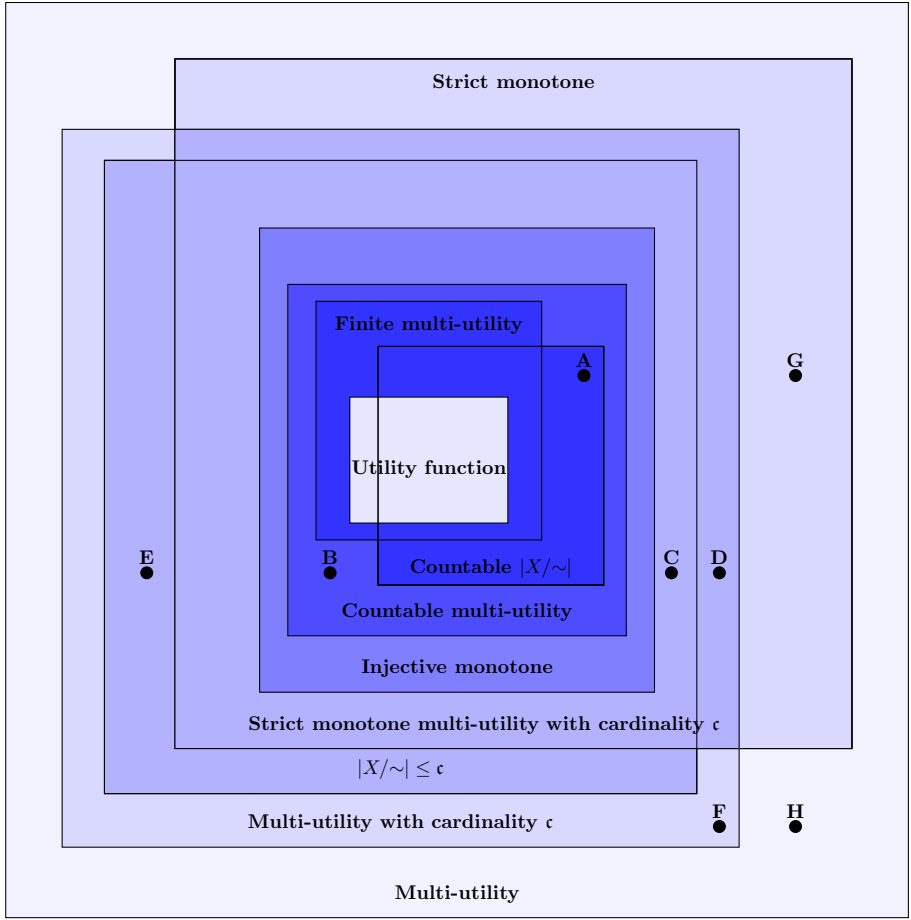

Consider a casino owner that intends to incorporate a new game to the casino, where all games under consideration follow the same idea, namely, bets are placed on the outcome of a random variable that is subsequently realized. Since the players win whenever they predict the outcome correctly, it is in the owner's interest to make the prediction as difficult as possible, that is, to make the game's outcome as uncertain as possible. For example, the game of rolling a fair die is preferred compared to that of rolling a loaded one. When considering the games over some finite set $\Omega$, the preference of the casino owner among them can be modeled on $\mathbb P_\Omega$,  the space of outcome probability distributions that are associated to the games,
by the uncertainty preorder
\begin{equation}
\label{uncert rela}
    p \preceq_U q \ \iff \ u_i(p) \leq u_i(q) \ \  \forall i\in \{1,..,|\Omega|-1\} \, ,
\end{equation}
where $u_i(p) \coloneqq -\sum_{n=1}^{i} p_n^{\downarrow}$ and $p^{\downarrow}$ denotes the decreasing rearrangement of $p$ (same components as $p$ but ordered decreasingly). Notice, $\preceq_U$ is known in mathematics, economics, and quantum physics as \emph{majorization} \cite{hardy1952inequalities,marshall1979inequalities,arnold2018majorization,brandao2015second}.

When deciding between certain game and another one, according to \eqref{uncert rela}, the owner evaluates $|\Omega|-1$ functions. Hence, the larger the number of possible outcomes, the harder it becomes to decide which game to choose. It is in this sense that we can say the number of functions  in \eqref{uncert rela} measures the \emph{complexity} or \emph{dimension} of the decision space.

In general, because of the gambling regulation, the owner may be required to choose games from some subset $B \subseteq \mathbb P_\Omega$ to diminish the gamblers' probability of loosing.
To automatize the decision, the owner uses the maximum entropy principle \cite{jaynes1957information,jaynes2003probability}, picking, hence, a distribution that maximizes Shannon entropy $H(p)\coloneqq -\mathbb E_p[\log p]$ over $B$. Although $H$ is not guaranteed to have maxima over every $B$, whenever it does, its maxima correspond to maximizing the preferences of the owner over $B$, that is, to distributions $p \in B$ such that there is no $q \in B$ that simultaneously fulfills $p \preceq_U q$ and $\neg(q \preceq_U p)$. Since the owner can make a decision on $B$ by simply optimizing $H$, we say $H$ is an \emph{optimization principle}. As we will see, the optimization properties of $H$ are closely related to how well $H$ preserves the properties of $\preceq_U$.

We introduce now the general picture in terms of both preorders and real-valued monotones, and return to optimization and complexity at the end of this section.
A \emph{preorder} $\preceq$ on a set $X$ is a reflexive ($x \preceq x$ $\forall x \in X$) and transitive ($x \preceq y$ and $ y \preceq z$ implies $x \preceq z$ $\forall x,y,z \in X$) binary relation. A tuple $(X, \preceq)$ is called a \emph{preordered space}  or \emph{preference} space and $X$ the \emph{ground set} or \emph{decision space}. For example, $(\mathbb P_\Omega,\preceq_U)$ is a preordered space. An antisymmetric ($x \preceq y$ and $y \preceq x$ imply $x=y$ $\forall x,y \in X$) preorder $\preceq$ is called a \emph{partial order}. The relation $x \sim y$, defined by $x\preceq y$ and $y\preceq x$, forms an \emph{equivalence relation} on $X$, that is, it fulfills the reflexive, transitive and symmetric ($x \sim y$ if and only if $y \sim x$ $\forall x,y \in X$) properties. Notice, a preorder $\preceq$ is a partial order on the quotient set $X/\mathord{\sim} = \{[x]|x\in X\}$, consisting of all equivalence classes $[x] = \{y\in X| y\sim x\}$. In case $x \preceq y$ and $\neg(x\sim y)$
for some $x,y \in X$ we say $y$ is \emph{strictly preferred} to $x$, denoted by $x \prec y$. If $\neg(x \preceq y)$ and $\neg(y \preceq x)$, we say $x$ and $y$ are \emph{incomparable}, denoted by $x \bowtie y$. Whenever there are no incomparable elements a preordered space is called \emph{total}. By the Szpilrajn extension theorem \cite{szpilrajn1930extension,harzheim2006ordered}, \textit{every partial order can be extended to a total order}, that is, to a partial order that is total.
Notice Szpilrajn extension theorem is a consequence of the axiom of choice, which we assume throughout this work. Equivalently, we assume $I \times I$ and $I$ are equinumerous for any infinite set $I$ and, thus, both $I \times \mathbb{N}$ and $I \cup I$ are also equinumerous to $I$.

In order to numerically characterize the relations established in a preordered space, one or several real-valued functions may be used. This results in a classification of preorders according to how well their information can be captured using these functions. We introduce now several classes that have been previously considered.  A real-valued function $f:X \rightarrow \mathbb{R}$ is called a \emph{monotone} or an \emph{increasing function} if $x \preceq y$ implies $f(x) \leq f(y)$ \cite{evren2011multi}. If the converse is also true, then $f$ is called a \emph{utility function} \cite{debreu1954representation}. Furthermore, if $f$ is a monotone and $x \prec y$ implies $f(x)<f(y)$, then $f$ is called a \emph{strict monotone}, a \emph{Richter-Peleg function} or an \emph{order-preserving function} \cite{alcantud2016richter}.
Similarly, a monotone $f$ is called an \emph{injective monotone}  if $f(x)=f(y)$ implies $x \sim y$, that is, if $f$ is injective considered as a function on the quotient set $X/\mathord{\sim}$ \cite{hack2022representing}.  For example, $H$ is a strict monotone but not an injective one on $(\mathbb P_\Omega,\preceq_U)$. Whenever a single function is insufficient to capture all the information in a preorder, for example when it is non-total (see \cite[Theorem 1.4.8]{bridges2013representations} for the total case), a family of functions may be used instead.
A family $V$ of real-valued functions $v: X \rightarrow \mathbb{R}$ is called a \emph{multi-utility (representation) of $\preceq$} \cite{evren2011multi} if 
\begin{equation*}
x \preceq y \iff v(x) \leq v(y)  \text{ }\forall v \in V \, . 
\end{equation*}
Whenever a multi-utility consists of strict monotones it is called a \emph{strict monotone} (or \emph{Richter-Peleg} \cite{alcantud2016richter}) \emph{multi-utility (representation) of $\preceq$}.  For example, $(u_i)_{i=1}^{|\Omega|-1}$ is a multi-utility that is not strict on $(\mathbb P_\Omega,\preceq_U)$. Analogously, if the multi-utility consists of injective monotones, we call it an \emph{injective monotone multi-utility (representation) of $\preceq$}. Notice the cardinality plays a key role in the classification when we consider multi-utilities.

The role of the different characterizations via a single function can be clarified alluding to optimization (see \cite[Section 4]{hack2022representing}). In this regard, monotones are not interesting in general, since they may carry no information about $\preceq$ (as in the case of constant functions). Strict monotones, however, are interesting from the perspective of optimization. In fact, strict monotones exist if and only if \emph{optimization principles} do \cite[Proposition 3]{hack2022representing}, where a function $f:X \to \mathbb R$ is an optimization principle or \emph{represents maximal elements} of $\preceq$ if, for any $B \subseteq X$, we have 
$\mathrm{argmax}_B  f \subseteq B^{\preceq}_M$, where $\mathrm{argmax}_B  f \coloneqq \{x \in B|
\nexists \text{ } y \in B \text{ such that } f(x)<f(y) \}$ and $B^{\preceq}_M \coloneqq \{ x \in B| \nexists \text{ } y \in B \text{ such that } x \prec y \}$.\footnote{Notice that the equivalence between the existence of optimization principles and that of strict monotones is not exactly what is stated in \cite[Proposition 3]{hack2022representing}, although it can be easily derived from it. The same holds true for injective monotones and injective optimization principles.} Hence, optimizing $f$ implies optimizing $\preceq$, like optimizing $H$ implies optimizing $\preceq_U$. Following this parallelism, the existence of injective monotones is equivalent to that of \emph{injective optimization principles} \cite[Proposition 3]{hack2022representing}, where a function $f:X \to \mathbb R$ is an injective optimization principle or \emph{injectively represents maximal elements} of $\preceq$ if, for any $B \subseteq X$ such that $\mathrm{argmax}_{B} f  \neq \emptyset$ , we have 
$\mathrm{argmax}_{B} f = [x_0]|_B$,
where $[x_0]|_B$ is the equivalence class of $x_0$ restricted to $B$. In particular, if $\preceq$ is a partial order, then optimizing $f$ yields a unique element, which is, ultimately, the goal of any optimization principle. Note that this is not the case in general for $H$, since it is not an injective monotone whenever $|\Omega| \geq 3$ \cite[Lemma 4]{hack2022representing}. However, the maximum entropy principle does output a single distribution in the cases where it is usually applied, namely, when $B \subseteq \mathbb P_{\Omega}$ is given by a linear constraint \cite{jaynes1957information,jaynes2003probability}. As a final remark, notice strict monotones allow the existence of \emph{local} injective optimization principles, that is, those where injectivity is fulfilled for some subset $B \subseteq X$ \cite{white1980notes}, which contrasts with the \emph{global} approach from injective monotones, which works for all subsets $B \subseteq X$. In summary, the existence of both strict and injective monotones can be characterized in terms of optimization principles.

The purpose of the different classes of characterizations via a family of functions or multi-utility can be clarified in terms of complexity, where we consider the sort of characterizations and the number of functions needed for each characterization as measures of complexity. Multi-utilities form the first class, where the aim is to find the minimal family of monotones $V$ such that any false preference, $\neg(x \preceq y)$, is contradicted, at least, by one monotone, $v(x)>v(y)$ for some $v \in V$.  Strict monotone multi-utilities form the second class, which can be characterized by those multi-utilities $V$ which fulfill $x \prec y \iff v(x) < v(y)  \text{ }\forall v \in V$. The third class, injective monotone multi-utilities, possesses the  properties of the previous two plus the fact $x \sim y \iff v(x)=v(y)$ for \emph{any} $v \in V$. Note, on the contrary, the previous two classes require that $v(x)=v(y)$ for all $v \in V$ to determine that $x \sim y$.

Inside each class, we can distinguish preorders in terms of the minimal amount of monotones that are required to form a multi-utility.
We will refer to this minimal amount, in the specific case of multi-utilities, as the \emph{dimension} of the preorder, since it characterizes the minimal number of copies of the real line that are needed, using their natural order, to fully represent a preorder. Note that this definition differs from the usual one by Dushnik and Miller \cite{dushnik1941partially}, since we restrict ourselves to products of the real line with its usual ordering instead of the general approach in \cite{dushnik1941partially} (see \cite{hack2022geometrical} for a discussion regarding the notion of dimension for partial orders). In fact, since their definition is restricted to partial orders, the distinction between  strict and injective monotone multi-utilities can be improved. Their definition considers a family of \emph{realizers} $(\preceq_i)_{i \in I}$, i.e. partial orders $\preceq_i$ that are total, fulfill that $x \preceq y$ implies $x \preceq_i y$ for all $i \in I$, and such that any false preference, $\neg(x \preceq y)$, is contradicted, at least, by one linear extension, $y \prec_i x$ for some $i \in I$. Such partial orders, however, cannot be defined using a strict monotone that is not injective $v$,
since there exist $x,y \in X$ such that $x \bowtie y$ and $v(x)=v(y)$. As a result, we obtain both $x \sim_v y$, given that we define $x \preceq_v y \iff v(x) \leq v(y)$, and $x \neq y$, contradicting, thus, antisymmetry.

Although several connections between the existence of these real-valued monotones are known \cite{evren2011multi,alcantud2016richter,hack2022representing}, we further clarify the relation between them throughout the following section. Mainly, using a characterization of these classes in terms of families of increasing sets which separate the elements in a preordered space \cite{herden1989existence,alcantud2013representations,bosi2013existence,hack2022representing}, we introduce several counterexamples which allow us to distinguish the scope of the different classes and, hence, to improve on the study of both complexity and optimization, and their relation in preordered spaces.

\section{Classification of preorders through real-valued monotones}
\label{sect: class}

\subsection{Characterization of real-valued monotones by families of increasing sets}

 A subset $A\subseteq X$ is called \emph{increasing}, if for all $x\in A$, $x\preceq y$ implies that $y\in A$ \cite{mehta1986existence}. We say a family $(A_i)_{i\in I}$ of subsets $A_i\subseteq X$ \emph{separates $x$ from $y$}, if there exists $i\in I$ with $x\not\in A_i$ and $y\in A_i$. Families of increasing sets have been used to characterize the existence of several classes of preorders in terms of real-valued representations in the literature. We state these results in Lemma \ref{set charac} without proof.
 
 \begin{lemma}
\label{set charac}
Let $(X,\preceq)$ be a preordered space.
\begin{enumerate}[label=(\roman*)]
\item For any infinite set $I$, there exists a multi-utility with the cardinality of $I$ if and only if there exists a family of increasing subsets $(A_i)_{i \in I}$
that $\forall x,y\in X$ with $x \prec y$ separates $x$ from $y$, and $\forall x,y\in X$ with $x \bowtie y$ separates both $x$ from $y$ and $y$ from $x$.
 \item There exists a strict monotone if and only if there exists a countable family of increasing subsets that $\forall x,y\in X$ with $x \prec y$ separates $x$ from $y$.
 \item There exists an injective monotone if and only if there exists a countable family of increasing subsets that $\forall x,y\in X$ with $x \prec y$ separates $x$ from $y$ and $\forall x,y\in X$ with $x \bowtie y$ separates either $x$ from $y$ or $y$ from $x$.
\end{enumerate}
\end{lemma}

The proof of $(i)$ can be found in \cite{bosi2013existence,alcantud2013representations} and that of $(ii)$ and $(iii)$ in \cite{hack2022representing}. The characterizations in Lemma \ref{set charac} can be useful to distinguish certain classes of preorders in terms of real-valued monotones, as we showed in \cite[Proposition 8]{hack2022representing}, where we used them to build a preorder where injective monotones exist and countable multi-utilities do not. Note that $(i)$ is not true for finite sets $I$, because there are preordered spaces that have a finite multi-utility but do not have a finite separating family of increasing subsets, for example majorization.


The statements in Lemma \ref{set charac} can be complemented with a characterization of the existence of strict monotone multi-utilities with the cardinality of an infinite set $I$, which we include in the following proposition.

\begin{proposition}
\label{R-P multi charac}
If $(X,\preceq)$ is a preordered space and 
$I$ is a set of infinite cardinality, then the following are equivalent.
\begin{enumerate}[label=(\roman*)]
\item There exists a strict monotone multi-utility with the cardinality of $I$.
\item There exist a strict monotone and a multi-utility with the cardinality of $I$.
\item There exists a family of increasing sets $(A_i)_{i \in I}$ which separates $x$ from $y$ if $x \bowtie y$
and a countable set $I' \subseteq I$ such that $(A_i)_{i \in I'}$
separates $x$ from $y$ if $x \prec y$.
\end{enumerate}
\end{proposition}

\begin{proof}
Clearly, $(i)$ implies $(ii)$ by definition. To show $(ii)$ implies $(iii)$, notice, given there is a multi-utility with the cardinality of $I$, we can follow the proof of Lemma \ref{set charac} $(i)$ in \cite{alcantud2013representations} to show there exists a family of increasing sets with the cardinality of $I \times \mathbb{N}$, which is equinumerous to $I$, which separates $x$ from $y$ whenever $x \bowtie y$. We can similarly follow the proof of Lemma \ref{set charac} $(ii)$ in \cite{hack2022representing} to get, given there exists a strict monotone, there exists a countable family of increasing sets which separates $y$ from $x$ whenever $x \prec y$. Since $I \cup \mathbb{N}$ is equinumerous to $I$, we get the desired result. In order to show $(iii)$ implies $(i)$, we can again follow both \cite{alcantud2013representations} and \cite[Proposition 7] {hack2022representing}. From the first one, we can construct a multi-utility $(u_i)_{i \in I}$ and, from the second one, we can construct a strict monotone $v$. Finally, we consider, as in \cite[Theorem 3.1]{alcantud2016richter}, the family of monotones $(v_{i,n})_{i \in I, n \in \mathbb{N}}$ where $v_{i,n} \coloneqq u_i + \alpha_n v$, where $(\alpha_n)_{n \in \mathbb{N}}$ is a numeration of the rational numbers which are greater than zero. This family can be shown to be a strict monotone multi-utility and has the cardinality of $I \times \mathbb{N}$, which is the same as that of $I$.
\end{proof}

Notice, in case the set $I$ is finite, the relation between $(i)$ and $(ii)$ is addressed in Proposition \ref{finite relations}.

\subsection{Improving the classification of preorders}
\label{imp class}

In this section, we present several results that improve on the classification of preordered spaces in terms of real-valued monotones. When applicable, we include, right after the proof of the result, an interpretation in terms of either complexity, optimization or the interplay between them.

Let us begin with the relation between preorders which have strict monotones and those which have injective monotones. Clearly, an injective monotone is also a strict monotone, since $x\prec y$ and $f(x)=f(y)$ contradicts injectivity. There are, however, preordered spaces with strict monotones and without injective monotones, as was shown in \cite[Proposition 1]{hack2022representing}. 
The argument there is
purely in terms of cardinality, since, whenever injective monotones exist, we have $|X/\mathord{\sim}| \leq \mathfrak{c}$ with $\mathfrak{c}$ the cardinality of the continuum, but there are preordered spaces with strict monotones and $|X/\mathord{\sim}|=|\mathcal{P}(\mathbb{R})|$. We can, however, improve upon this by showing there are preordered spaces where $X/\mathord{\sim}$ has the cardinality of the continuum and strict monotones exist while injective monotones do not. We include such a preordered space in Proposition \ref{strict mono no injective}.


\begin{proposition}
\label{strict mono no injective}
There are preordered spaces $(X,\preceq)$ where $X/\mathord{\sim}$ has the cardinality of the continuum $\mathfrak{c}$ and 
strict monotones exist while injective monotones do not.
\end{proposition}

\begin{proof}
Consider $X \coloneqq [0,1] \cup [2,3]$ equipped with $\preceq$ where 
\begin{equation}
\label{order def 2}
    x \preceq y \iff 
    \begin{cases}
    x,y \in [0,1] \text{ and } x \leq y,\\
    x,y \in [2,3] \text{ and } x \leq y,\\
    x  \in [0,1],\text{ } y \in [2,3] \text{ and } x+2<y,\\
    x  \in [2,3],\text{ } y \in [0,1] \text{ and } x-2 <y
    \end{cases}
    \end{equation}
    $\forall x,y \in X$ (see Figure \ref{fig 2} for a representation of $\preceq$). Notice $(X,\preceq)$ is a preordered space and $v: X \to \mathbb{R}$ where $x \mapsto x$ if $x \in [0,1]$ and $x \mapsto x-2$ if $x \in [2,3]$ is a strict monotone.  We will show that any family $(A_i)_{i \in I}$, where $A_i \subseteq X$ is increasing $\forall i \in I$ and $\forall x,y \in X$ such that $x \bowtie y$ there exists some $i \in I$ such that either $x \not \in A_i$ and $y \in A_i$ or $y \not \in A_i$ and $x \in A_i$, is uncountable. Since the existence of some $(A_i)_{i \in I}$ with these properties and countable $I$ is implied by the existence of an injective monotone by Lemma \ref{set charac} $(iii)$, we obtain that there is no injective monotone for $X$.
    
   Let $(A_i)_{i \in I}$ be a family with the properties in the last paragraph and, for each $x \in [0,1]$, define $y_x \coloneqq x+2$. Since $x \bowtie y_x$ by definition, there exists some $A_x \in (A_i)_{i \in I}$ such that either $x \in A_x$ and $y_x \not \in A_x$ or $y_x \in A_x$ and $x \not \in A_x$. We fix such an $A_x$ for each $x \in [0,1]$ and consider the map $f:[0,1] \to (A_i)_{i \in I}$, $x \mapsto A_x$. Consider some $x,z \in [0,1]$ such that $A_x=A_z$ and assume $x \neq z$. We show first the case where $z < x$ leads to contradiction. Assume first $x \in A_x$ and $y_x \not \in A_x$. Then, since $A_x=A_z$, we either have $z \in A_x$ or $y_z \in A_x$. Both cases lead to contradiction, since we get $y_x \in A_x$ because $A_x$ is increasing and we either have $z \prec y_x$ with $z \in A_x$ or $y_z \prec y_x$ with $y_z \in A_x$. We can proceed analogously if we assume $y_x \in A_x$, relying on the fact both $z \prec x$ and $y_z \prec x$ hold.
    In case we assume $x <z$, we also achieve a contradiction following the same argument but interchanging the role of $x$ and $z$.
    Thus, $x \neq z$ leads to a contradiction and by injectivity of $f$ we get $|[0,1]| \leq |(A_i)_{i \in I}|$. As a consequence, $X$ has no injective monotone.
\end{proof}

\begin{figure}[!tb]
\centering
\begin{tikzpicture}
\node[rounded corners, draw,fill=blue!10, text height = 3cm, minimum width = 11cm,xshift=4cm,label={[anchor=west,left=.1cm]180:\textbf{B}}] {};
\node[rounded corners, draw,fill=blue!10, text height = 3cm, minimum width = 11cm,xshift=4cm,yshift=-4cm,label={[anchor=west,left=.1cm]180:\textbf{A}}] {};
    \node[main node] (1) {$x+2$};
    \node[main node] (2) [right = 2cm  of 1]  {$y+2$};
    \node[main node] (3) [right = 2cm  of 2]  {$z+2$};
    \node[main node] (4) [below = 2cm  of 1] {$x$};
    \node[main node] (5) [right = 2cm  of 4] {$y$};
    \node[main node] (6) [right = 2cm  of 5] {$z$};

    \path[draw,thick,->]
    (4) edge node {} (2)
    (4) edge node {} (3)
    (5) edge node {} (3)
    (4) edge node {} (5)
    (5) edge node {} (6)
    (1) edge node {} (2)
    (2) edge node {} (3)
    (1) edge node {} (5)
    (1) edge node {} (6)
    (2) edge node {} (6)
    (1) edge [bend left] node {} (3)
    (4) edge [bend right] node {} (6)
    ;
\end{tikzpicture}
\caption{Graphical representation of a preordered space, defined in Proposition \ref{strict mono no injective}, with the cardinality of the continuum and where strict monotones exist while injective monotone do not. In particular, we show $A \coloneqq [0,1]$, $B \coloneqq [2,3]$ and how $x,y,z \in A$, $x<y<z$, are related to $x+2,y+2,z+2 \in B$. Notice, an arrow from an element $w$ to an element $t$ represents $w \prec t$.}
\label{fig 2}
\end{figure}
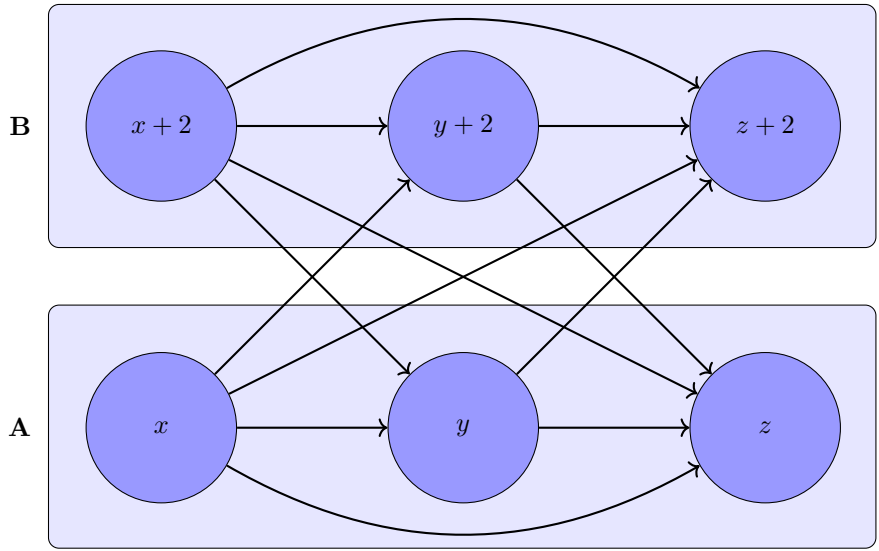

Proposition \ref{strict mono no injective} shows there are preference spaces for which optimization principles exist while injective ones do not. In particular, it shows it is sufficient to consider preference spaces with a continuum decision space to find such cases (notice, as shown in \cite[Proposition 5]{hack2022representing}, preference spaces with a countable ground set always have injective optimization principles). Furthermore, it shows that, unlike when the decision space is countable, the existence of local injective optimization principles and that of global ones are not equivalent. Hence, provided the decision space is sufficiently large, differences between local and global optimization arise.
Notice, the proof of Proposition \ref{strict mono no injective} relies on the existence of connections between several elements, which allow us to assure the sets from Lemma \ref{set charac} that separate the preorder differ when different elements inside certain sets are considered. This is the reason why a related preorder was used in \cite[Proposition 8]{hack2022representing} to show
countable multi-utilities and injective monotones are not equivalent. One may think the trivial preorder on the real line
$(\mathbb{R},=)$ would have an injective monotone, the identity, and no countable multi-utility.\footnote{We say a binary relation $\preceq$ on a set $X$ is a \emph{trivial ordering} if $x \preceq y \iff x=y$ $\forall x,y \in X$.} However, due to it being completely disconnected, $(\chi_{\leq q},\chi_{\geq q})_{q \in \mathbb{Q}}$ is a countable multi-utility, where $\chi_{\leq q}(x) \coloneqq 1$ if $x \leq q$ and $\chi_{\leq q}(x) \coloneqq 0$ otherwise, and $\chi_{\geq q}(x) \coloneqq 1$ if $x \geq q$ and $\chi_{\geq q}(x) \coloneqq 0$ otherwise.

While the existence of injective monotones implies the existence of multi-utilities with the cardinality of the continuum (in particular, composed of injective monotones), as we showed in \cite[Proposition 4]{hack2022representing}, the converse was unknown up to now. The preordered space in Proposition \ref{strict mono no injective} shows the converse in false. Actually, it shows the stronger statement that the existence of strict monotone multi-utilities with cardinality $\mathfrak{c}$ is still not sufficient for the existence of an injective monotone, as we state in Corollary \ref{R multi-ut no inj mono}.

\begin{corollary}
\label{R multi-ut no inj mono}
There are preordered spaces which have 
strict monotone multi-utilities with cardinality $\mathfrak{c}$ and no injective monotone.
\end{corollary}

\begin{proof}
We can use the counterexample from Proposition \ref{strict mono no injective} which has no injective monotone. Moreover, it is straightforward to see that $(\chi_{i(x)})_{x \in X}$ is a multi-utility with cardinality $\mathfrak{c}$  \cite{evren2011multi}, where $\chi_A$ is the indicator function of a set $A$ and $i(x) \coloneqq \{y \in X| x \preceq y \}$ $\forall x \in X$. Since there exist strict monotones, as we showed in the proof of Proposition \ref{strict mono no injective}, we can follow Proposition \ref{R-P multi charac} and get that there exist 
strict monotone multi-utilities with cardinality $\mathfrak{c}$.
\end{proof}

Corollary \ref{R multi-ut no inj mono} states, by Proposition \ref{R-P multi charac}, that the existence of optimization principles for preference spaces with continuum dimension is not enough for injective optimization principles to exist. This differs from the case of countable complexity, where injective optimization principles always exist \cite{hack2022representing}. Thus, whenever the preference space is sufficiently involved, local and global injective optimization principles do not coincide in general.
Notice, Corollary \ref{R multi-ut no inj mono} implies the class of preorders with injective monotones is strictly contained inside the class where multi-utilities with cardinality $\mathfrak{c}$ exist. In fact, we can improve upon this modifying the preorder in Proposition \ref{strict mono no injective} to show there are preordered spaces where multi-utilities with cardinality $\mathfrak{c}$ exist while strict monotones do not. We present such a preorder in Proposition \ref{R muli-uti no strict mono}, which is the same as the one in Proposition \ref{strict mono no injective} with the exception that we have $x \prec y_x$ instead of $x \bowtie y_x$ $\forall x \in [0,1]$.

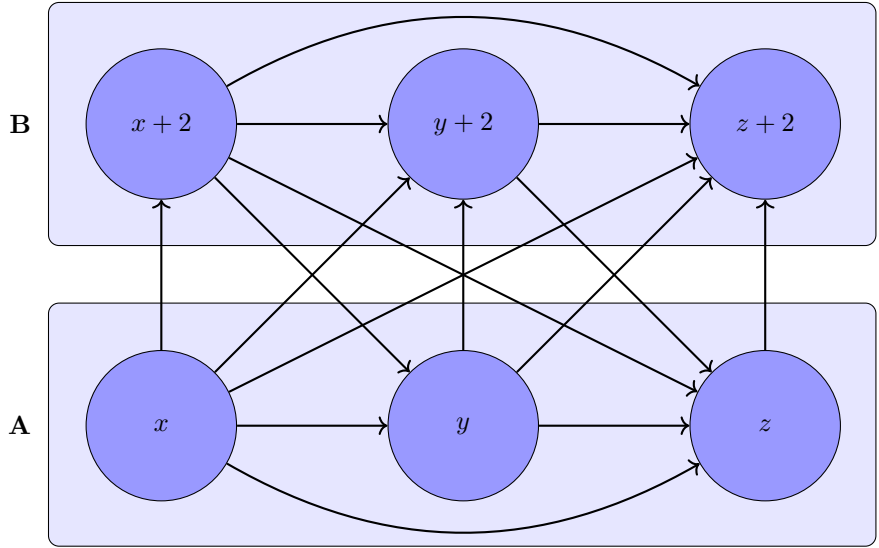
\begin{figure}[!tb]
\centering
\begin{tikzpicture}
\node[rounded corners, draw,fill=blue!10, text height = 3cm, minimum width = 11cm,xshift=4cm,label={[anchor=west,left=.1cm]180:\textbf{B}}] {};
\node[rounded corners, draw,fill=blue!10, text height = 3cm, minimum width = 11cm,xshift=4cm,yshift=-4cm,label={[anchor=west,left=.1cm]180:\textbf{A}}] {};
    \node[main node] (1) {$x+2$};
    \node[main node] (2) [right = 2cm  of 1]  {$y+2$};
    \node[main node] (3) [right = 2cm  of 2]  {$z+2$};
    \node[main node] (4) [below = 2cm  of 1] {$x$};
    \node[main node] (5) [right = 2cm  of 4] {$y$};
    \node[main node] (6) [right = 2cm  of 5] {$z$};

    \path[draw,thick,->]
    (4) edge node {} (1)
    (4) edge node {} (2)
    (4) edge node {} (3)
    (5) edge node {} (3)
    (6) edge node {} (3)
    (1) edge node {} (5)
    (2) edge node {} (6)
    (1) edge node {} (6)
    (5) edge node {} (2)
    (4) edge node {} (5)
    (5) edge node {} (6)
    (1) edge node {} (2)
    (2) edge node {} (3)
    (1) edge [bend left] node {} (3)
    (4) edge [bend right] node {} (6)
    ;
\end{tikzpicture}
\caption{Graphical representation of a preordered space, defined in Proposition \ref{R muli-uti no strict mono}, which has multi-utilities with  cardinality $\mathfrak{c}$ and no strict monotone. In particular,  we show $A \coloneqq [0,1]$, $B \coloneqq [2,3]$ and how $x,y,z \in A$, $x<y<z$, are related to $x+2,y+2,z+2 \in B$. Notice, an arrow from an element $w$ to an element $t$ represents $w \prec t$.}
\label{fig 4}
\end{figure}

\begin{proposition}
\label{R muli-uti no strict mono}
There are prerordered spaces which have multi-utilities with cardinality $\mathfrak{c}$ and no strict monotone.
\end{proposition}

\begin{proof}
Consider $X \coloneqq [0,1] \cup [2,3]$ equipped with $\preceq$ where 
\begin{equation}
\label{order def 3}
x \preceq y \iff 
\begin{cases}
    x,y \in [0,1] \text{ and } x \leq y,\\
    x,y \in [2,3] \text{ and } x \leq y,\\
x  \in [0,1],\text{ } y \in [2,3] \text{ and } x+2 \leq y\\
x  \in [2,3],\text{ } y \in [0,1] \text{ and } x-2 <y
\end{cases}
\end{equation}
$\forall x,y \in X$ (see Figure \ref{fig 4} for a representation of $\preceq$), which differs from \eqref{order def 2} only in $x+2\leq y$ instead of $x+2< y$ for $x\in [0,1]$ and $y\in [2,3]$. Notice $(X,\preceq)$ is a preordered space and there is a multi-utility with cardinality $\mathfrak{c}$ as in the proof of Corollary \ref{R multi-ut no inj mono}. We will show that any family $(A_i)_{i \in I}$, where $A_i \subseteq X$ is increasing $\forall i \in I$ and $\forall x,y \in X$ such that $x \prec y$ there exists some $i \in I$ such that $y \in A_i$ and $x \not \in A_i$, is uncountable. Since the existence of some $(A_i)_{i \in I}$ with these properties and countable $I$ is implied by the existence of a strict monotone by Lemma \ref{set charac} $(ii)$, we conclude that there is no strict monotone for $X$.

Let $(A_i)_{i \in I}$ be a family with the properties in the last paragraph and, for each $x \in [0,1]$, define $y_x \coloneqq x+2$. Since $x \prec y_x$ by definition, there exists some $A_x \in (A_i)_{i \in I}$ such that both $y_x \in A_x$ and $x \not \in A_x$ hold. We fix such an $A_x$ for each $x \in [0,1]$ and consider the map $f:[0,1] \to (A_i)_{i \in I}$, $x \mapsto A_x$. Consider some $x,z \in [0,1]$ such that $A_x=A_z$ and assume $x \neq z$. We  show first the case where $z < x$ leads to contradiction. Since $A_x=A_z$, we have $y_z \in A_x$. Given the fact $A_x$ is increasing and $y_z \prec x$ by definition, we get $x \in A_x$, a contradiction. In case we assume $x <z$, we also achieve a contradiction following the same argument but interchanging the role of $x$ and $z$. Thus, $x \neq z$ leads to contradiction and we get, by injectivity of $f$, $|[0,1]| \leq |(A_i)_{i \in I}|$. As a consequence, $X$ has no strict monotone.
\end{proof}

The relation between optimization and complexity is improved in Proposition \ref{R muli-uti no strict mono}, where we show that having a continuum dimension is not sufficient for optimization principles to exist, which contrasts with the case where the dimension is countable \cite{alcantud2016richter}. In summary, no optimization principle may exist provided the complexity of the preference space is large enough (see, also, Corollary \ref{|A| MU no R-P}).
Notice, essentially, we recover in Proposition \ref{R muli-uti no strict mono} the \emph{lexicographic plane}, the classical counterexample used by Debreu \cite{debreu1954representation,debreu1959theory} to show the existence of total preordered spaces without utility functions. Another counterexample, which relies on Szpilrajn extension theorem, can be found in \cite[A.2.1]{hack2022representing}.  Notice, in particular, Proposition \ref{R muli-uti no strict mono} implies that the class of preordered spaces with strict monotone multi-utilities with cardinality $\mathfrak{c}$ is strictly contained inside the class with multi-utilities of the same cardinality. This contrasts with the fact that countable multi-utilities and countable strict monotone multi-utilities coincide for any preordered space \cite[Proposition 4.1]{alcantud2016richter}. In fact, they also coincide with countable injective monotone multi-utilities \cite[Proposition 6]{hack2022representing}. Notice, also, the preordered space in Proposition \ref{R muli-uti no strict mono} shows the stronger fact
that strict monotones do not always exist when $X/\mathord{\sim}$ has cardinality $\mathfrak{c}$, as we state
in Corollary \ref{cardin no strict mono}.

\begin{corollary}
\label{cardin no strict mono}
There are preordered spaces $(X,\preceq)$ where $X/\mathord{\sim}$ has cardinality $\mathfrak{c}$
and strict monotone multi-utilities with cardinality $\mathfrak{c}$ do not exist.
\end{corollary}

\begin{proof}
Consider the preordered space in Proposition \ref{R muli-uti no strict mono}. Notice, since $(\chi_{i(x)})_{x \in X/\mathord{\sim}}$ is a multi-utility of cardinality $\mathfrak{c}$, strict monotone multi-utilities of cardinality $\mathfrak{c}$ and strict monotones are equivalent, by Proposition \ref{R-P multi charac}. Thus, they do not exist.
\end{proof}

Corollary \ref{cardin no strict mono} states that we can strengthen the bound in Proposition \ref{R muli-uti no strict mono} from the dimension to the decision space and, nonetheless, optimization principles do not exist. That is, optimization principles may not exist provided the amount of alternatives in the decision space is sufficiently large.
Notice, if $X/\mathord{\sim}$ is countable, then it has countable multi-utilities (we can follow the proof in Corollary \ref{R multi-ut no inj mono}) and, by \cite[Theorem 3.1]{alcantud2016richter}, countable strict monotone multi-utilities. Furthermore, we can follow 
Corollary \ref{R multi-ut no inj mono} and Proposition \ref{R-P multi charac} to conclude that every preorder with strict monotones 
has strict monotone multi-utilities with the cardinality of some infinite set $I$ if $X/\mathord{\sim}$ has the cardinality of $I$. As we show in Proposition \ref{mulit-ut and cardi}, the converse is not true, that is, whenever multi-utilities with the cardinality of an infinite set $I$ exist, we have $|X/\mathord{\sim}| \leq |\mathcal{P}(I)|$ with some preorders achieving equality. Furthermore, the bound cannot be improved even when strict monotone multi-utilities with the cardinality of $I$ exist. Equivalently, we show whenever an infinite Debreu upper dense subset $I \subseteq X$ exists, then we have $w(X,\preceq) \leq |\mathcal{P}(I)|$ where $w(X,\preceq)$ is the \emph{width} of $(X,\preceq)$, that is, the maximal cardinality of the antichains\footnote{Any two elements in an antichain are incomparable.} in $X$. Recall
we say a subset $Z\subseteq X$ is \emph{upper dense in the sense of Debreu} (or \emph{Debreu upper dense} for short) if $x \bowtie y$ implies that there exists a $z \in Z$ such that $x \bowtie z \preceq y$ \cite{hack2022representing}.\footnote{Notice, for a fixed pair $x,y \in X$ where $x \bowtie y$ holds, there exist $a_1,a_2 \in Z$ such that $x \bowtie a_1 \preceq y$ and $y \bowtie a_2 \preceq x$.}

\begin{proposition}
\label{mulit-ut and cardi}
Let $(X,\preceq)$ be a preordered space and $I$ be an infinite set.

\begin{enumerate}[label=(\roman*)]
\item If there exist multi-utilities with the cardinality of $I$, then $|X/\mathord{\sim}| \leq |\mathcal{P}(I)|$, where $\mathcal P(I)$ denotes the power set of $I$. Furthermore, the bound is sharp, i.e. it cannot be improved.
\item  Even if there exist strict monotone multi-utilities with the cardinality of $I$, the bound in $(i)$ is sharp.
\item If $I \subseteq X$ is a Debreu upper dense subset, then $w(X,\preceq) \leq |\mathcal{P}(I)|$, where $w(X,\preceq)$ is the width of $(X,\preceq)$. Furthermore, the bound is sharp.
\end{enumerate}
\end{proposition}

\begin{proof}
$(i)$ For the first statement, notice, by Lemma \ref{set charac} $(i)$, there exists a family of increasing sets $(A_i)_{i \in I}$ that $\forall x,y\in X$ with $x \prec y$ separates $x$ from $y$ and $\forall x,y\in X$ with $x \bowtie y$ separates both $x$ from $y$ and $y$ from $x$.
Consider the map $f: X/\mathord{\sim} \to \mathcal{P}(I)$, $[x] \mapsto B_x$ where $B_x \coloneqq \{ i \in I| [x] \subseteq A_i\}$. If $[x] \neq [y]$, then we either have $x \bowtie y$, $x \prec y$ or $y \prec x$ $\forall x \in [x]$, $ y \in [y]$. In any case, there exists some $i \in I$ such that $x \subseteq A_i$ and $y \not \subseteq A_i$ and vice versa. Thus, $B_x \neq B_y$ and $f$ is injective. We get $|X/\mathord{\sim}| \leq |\mathcal{P}(I)|$. 

For the second statement, consider the set $X \coloneqq \mathcal{P}(I)$ equipped with the preorder $\subseteq$, where $\subseteq$ denotes set inclusion. One can see $(f_i)_{i \in I}$ is a multi-utility for $X$, where $f_i: \mathcal{P}(I) \to \mathbb{R}$, $U \mapsto 1$ if $i \in U$ and $U \mapsto 0$ otherwise. Notice
we have $|\mathcal{P}(I)/\mathord{\sim}| = |\mathcal{P}(I)|$. Thus, the bound in the first statement cannot be improved.

$(ii)$ Consider $X \coloneqq \mathcal{P}(I)$ equipped with the  
trivial ordering $\preceq$.
Notice $(f_i)_{i \in I} \cup (g_i)_{i \in I}$ is a strict monotone multi-utility with the cardinality of $I \cup I$, which is equinumerous to $I$, where $f_i: X \to \mathbb{R}$, $U \mapsto 1$ if $i \in U$ and $U \mapsto 0$ otherwise and $g_i \coloneqq -f_i$, and we also have $|\mathcal{P}(I)/\mathord{\sim}| = |\mathcal{P}(I)|$. Thus, the bound in $(i)$ cannot be improved.

$(iii)$ Consider $A$ an antichain of $X$ and, for each $x \in A$, $I_x \coloneqq \{i \in I| i \preceq x\}$. We will show the map $f:A \to \mathcal{P}(I)$, $x \mapsto I_x$ is injective, proving, thus, any antichain $A$ fulfills $|A| \leq |\mathcal{P}(I)|$ which leads to $w(X,\preceq)\leq |\mathcal{P}(I)|$. Given $x,y \in A$, $x \neq y$, we have $x \bowtie y$ and, by Debreu upper density of $I$, there exists some $i \in I$ such that $x \bowtie i \preceq y$. As a consequence, $i \in I_y$ and $i \not \in I_x$. Resulting in $f(x)=I_x \not = I_y=f(y)$ and, hence, in $f$ being injective.

For the second statement, consider the preorder $(\Sigma^* \cup \Sigma^\omega,\preceq)$ where $\Sigma \coloneqq \{0,1\}$, $ \Sigma^*$ is the set of finite sequences over $\Sigma$, $\Sigma^\omega$ is the set of infinite sequences over $\Sigma$, if $x  \in \Sigma^*$ and $y \in \Sigma^\omega$ then $x \preceq_C y$ if $x$ is a prefix of $y$ and $\preceq$ is defined $\forall x,y \in \Sigma^* \cup \Sigma^\omega$ like
\begin{equation*}
    x \preceq y \iff 
    \begin{cases}
    x=y\\
    x \preceq_C y.
    \end{cases}
\end{equation*}
Notice $\Sigma^*$ is a countable Debreu upper dense subset and we have $w(X,\preceq) = |\Sigma^w|=|\mathcal{P}(\Sigma^*)|$. Thus, the bound in the first statement cannot be improved.
\end{proof}

Primarily, Proposition \ref{mulit-ut and cardi} shows how, whenever they are infinite, bounds on the dimension of a preference space result in bounds on the number of commodities in the decision space and, moreover, it shows that these bounds are optimal, since they are achieved by some preference spaces.
Notice, while the analog of $(iii)$ remains a question whenever $I$ is a finite set, the bounds in both $(i)$ and $(ii)$ do not hold. Although a trivial example supporting this assertion would be the real line with its usual order $(\mathbb{R},\leq)$, since the identity is a strict monotone finite multi-utility, we conclude this paragraph including two, perhaps, more interesting counterexamples. Majorization also proves $(i)$ is false when $I$ is finite, since, although it is defined through \eqref{uncert rela}, it fulfills $|\mathbb P_\Omega/\mathord{\sim_U}|=\mathfrak{c}$. We adapt from \cite[Example 2]{bosi2020topologies} a preorder which also illustrates that $(ii)$ is false when $I$ is finite. In particular, we take $X\coloneqq A \cup B$, where $A$ and $B$ are two copies of $\mathbb{R}/\{0\}$, and equip them with $\preceq$ where
\begin{equation}
\label{finite strict counter}
    x \preceq y \iff 
    \begin{cases}
    x,y \in A \text{ and } x \leq y,\\
    x,y \in B \text{ and } x \leq y,\\
    x  \in A, \text{ } x <0, \text{ } y \in B \text{ and } 0<y,\\
   x  \in B, \text{ } x <0, \text{ } y \in A \text{ and } 0<y
    \end{cases}
    \end{equation}
    $\forall x,y \in X$ (see Figure \ref{last fig} for a representation of $\preceq$). Note $|X/\mathord{\sim}|=\mathfrak{c}$ and $V \coloneqq \{v_1,v_2\}$ is a finite strict monotone multi-utility, where $v_1(x) \coloneqq x-1$ if $x \in A$ and $x<0$, $v_1(x) \coloneqq e^x-1$ if $x \in B$ and $x<0$, $v_1(x) \coloneqq 1-e^{-x}$ if $x \in B$ and $x>0$ and $v_1(x) \coloneqq x+1$ if $x \in A$ $x>0$, and $v_2(x) \coloneqq x-1$ if $x \in B$ and $x<0$, $v_2(x) \coloneqq e^x-1$ if $x \in A$ and $x<0$, $v_2(x) \coloneqq 1-e^{-x}$ if $x \in A$ and $x>0$ and $v_2(x) \coloneqq x+1$ if $x \in B$ and $x>0$.
    
    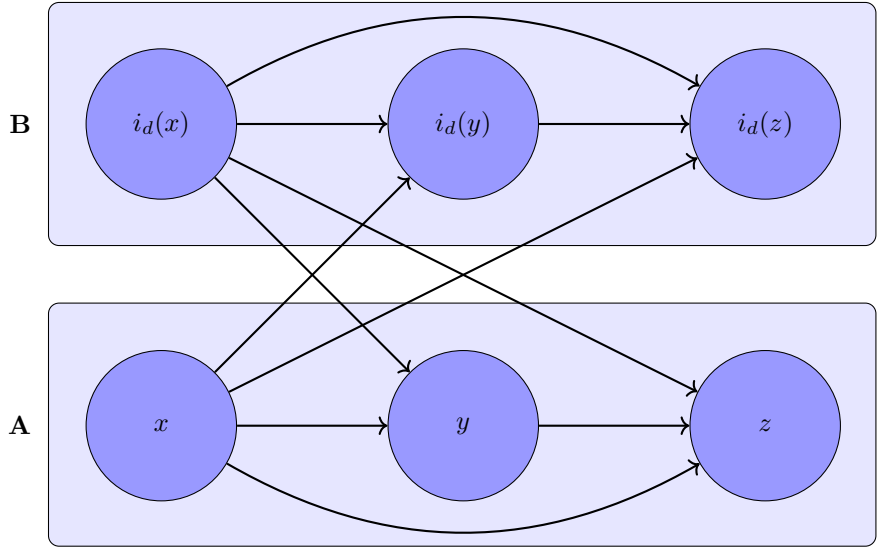
\begin{figure}[!tb]
\centering
\begin{tikzpicture}
\node[rounded corners, draw,fill=blue!10, text height = 3cm, minimum width = 11cm,xshift=4cm,label={[anchor=west,left=.1cm]180:\textbf{B}}] {};
\node[rounded corners, draw,fill=blue!10, text height = 3cm, minimum width = 11cm,xshift=4cm,yshift=-4cm,label={[anchor=west,left=.1cm]180:\textbf{A}}] {};
    \node[main node] (1) {$i_d(x)$};
    \node[main node] (2) [right = 2cm  of 1]  {$i_d(y)$};
    \node[main node] (3) [right = 2cm  of 2]  {$i_d(z)$};
    \node[main node] (4) [below = 2cm  of 1] {$x$};
    \node[main node] (5) [right = 2cm  of 4] {$y$};
    \node[main node] (6) [right = 2cm  of 5] {$z$};

    \path[draw,thick,->]
    (4) edge node {} (2)
    (4) edge node {} (3)
    (4) edge node {} (5)
    (5) edge node {} (6)
    (2) edge node {} (3)
    (1) edge [bend left] node {} (3)
    (1) edge node {} (5)
    (1) edge node {} (6)
    (1) edge node {} (2)
    (4) edge [bend right] node {} (6)
    ;
\end{tikzpicture}
\caption{Representation of a preordered space, defined by \eqref{finite strict counter}, where finite strict monotone multi-utilities exist and $|X/\mathord{\sim}|=\mathfrak{c}$. In particular, we relate three different points $x,y,z \in A$ with $i_d(x),i_d(y),i_d(z) \in C$, where $A,B \coloneqq \mathbb{R}/\{0\}$, $x<0<y<z$ and $i_d:A\to B$ is the identity on $\mathbb{R}/\{0\}$. Notice an arrow from an element $w$ to an element $t$ represents $w \prec t$.}
\label{last fig}
\end{figure}

Proposition \ref{mulit-ut and cardi} improves the relation between the existence of multi-utilities and the cardinality of $X/\mathord{\sim}$. In particular, whenever we have $|X/\mathord{\sim}| \leq \mathfrak{c}$, then there exist multi-utilities with cardinality $\mathfrak{c}$ \cite{evren2011multi} (see Corollary \ref{R multi-ut no inj mono}) and, whenever injective monotones exist, we have $|X/\mathord{\sim}| \leq \mathfrak{c}$. However, there are preorders where, although $|X/\mathord{\sim}| \leq \mathfrak{c}$ holds, injective monotones do not exist, like the one in Proposition \ref{strict mono no injective}. Furthermore, there are preorders with $|X/\mathord{\sim}| \leq \mathfrak{c}$ where strict monotones do not exist, like the example in Corollary \ref{cardin no strict mono}. Finally, there exist preorders with strict monotone multi-utilities with cardinality $\mathfrak{c}$ and where $|X/\mathord{\sim}| > \mathfrak{c}$, like the one in Proposition \ref{mulit-ut and cardi} $(ii)$.

Returning to Proposition \ref{R muli-uti no strict mono}, notice its converse also holds, that is, there are preordered spaces where strict monotones exist and multi-utilities with cardinality $\mathfrak{c}$ do not. In general, for any uncountable set $I$, there exist preordered spaces where stict monotones exist and multi-utilities with the cardinality of $I$ do not, as we show in Proposition 
\ref{strict mono no R multi}. Notice the counterexample we present is, essentially, the one we introduced in \cite[Proposition 8]{hack2022representing}, but for a larger ground set. Despite the large ground set, however, the proof is constructive.

\begin{proposition}
\label{strict mono no R multi}
If $I$ is an uncountable set, then there exist preordered spaces with strict monotones and without multi-utilities with the cardinality of $I$.
\end{proposition}

\begin{proof}
Consider $X \coloneqq B \cup C$, where $B$ and $C$ are two copies of $\mathcal{P}(I)$, equipped with $\preceq$ where
\begin{equation}
        x \preceq y \iff 
    \begin{cases}
    x=y\\
    x  \in B,\text{ } y \in C \text{ and } y \neq i_d(x)
    \end{cases}
    \end{equation}
    $\forall x,y \in X$ with $i_d:B \to C$ the identity on $\mathcal{P}(I)$ (see Figure \ref{fig 44} for a representation of $\preceq$). Notice $(X,\preceq)$ is a preordered space and $v: X \to \mathbb{R}$ $x \mapsto 0$ if $x \in B$ and $x \mapsto 1$ if $x \in C$ is a strict monotone.  By Lemma \ref{set charac} $(i)$ there exists a family $(A_j)_{j \in J}$ of increasing subsets of $X$ such that whenever $x \bowtie y$ there exists some $j \in J$ such that $x \in A_j$ and $y \not \in A_j$. It is enough to show that such a family has larger cardinality than $I$ in order to see that there is no multi-utility for $X$ with the cardinality of $I$.

Notice, $x \bowtie i_d(x)$  $\forall x \in B$. There exists, thus, some $A_x \in (A_j)_{j \in J}$ such that $x \in A_x$ and $i_d(x) \not \in A_x$. We fix such an $A_x$ for each $x \in B$ and consider the map $f:B \to (A_j)_{j \in J}$, $x \mapsto A_x$. Consider a pair $x,z \in B$ such that $A_x =A_z$ and assume $x \neq z$. Since $A_x$ is increasing, $z \prec i_d(x)$ and $z \in A_x$, we get $i_d(x) \in A_x$, a contradiction. Thus, $A_x=A_z$ implies $x=z$ and we have, by injectivity of $f$, $|\mathcal{P}(I)|=|B| \leq |(A_j)_{j \in J}|$. As a consequence, $X$ has no multi-utility with the cardinality of $I$.
\end{proof}

Proposition \ref{strict mono no R multi} shows that the existence of optimization principles does not imply any bound on the dimension of the preference space in question. This does vary with respect to the case where injective optimization principles exist since, there, the dimension cannot surpass the continuum. Hence, although injective optimization is not, optimization is possible in spaces of arbitrarily large complexity.
\begin{figure}[!tb]
\centering
\begin{tikzpicture}
\node[rounded corners, draw,fill=blue!10, text height = 3cm, minimum width = 11cm,xshift=4cm,label={[anchor=west,left=.1cm]180:\textbf{C}}] {};
\node[rounded corners, draw,fill=blue!10, text height = 3cm, minimum width = 11cm,xshift=4cm,yshift=-4cm,label={[anchor=west,left=.1cm]180:\textbf{B}}] {};
    \node[main node] (1) {$i_d(x)$};
    \node[main node] (2) [right = 2cm  of 1]  {$i_d(y)$};
    \node[main node] (3) [right = 2cm  of 2]  {$i_d(z)$};
    \node[main node] (4) [below = 2cm  of 1] {$x$};
    \node[main node] (5) [right = 2cm  of 4] {$y$};
    \node[main node] (6) [right = 2cm  of 5] {$z$};

    \path[draw,thick,->]
    (4) edge node {} (2)
    (4) edge node {} (3)
    (5) edge node {} (1)
    (5) edge node {} (3)
    (6) edge node {} (1)
    (6) edge node {} (2)
    ;
\end{tikzpicture}
\caption{Representation of a preordered space, defined in Proposition \ref{strict mono no R multi}, where strict monotones exist and multi-utilities with the cardinality of $I$, an uncountable set, do not. In particular, we relate three different points $x,y,z \in B$ with $i_d(x),i_d(y),i_d(z) \in C$, where $B,C \coloneqq \mathcal{P}(I)$ and and $i_d:B\to C$ is the identity on $\mathcal{P}(I)$. Notice an arrow from an element $w$ to an element $t$ represents $w \prec t$. Notice, also, this preorder is, essentially, the one we introduced in \cite[Proposition 8]{hack2022representing} with a larger ground set.}
\label{fig 44}
\end{figure}
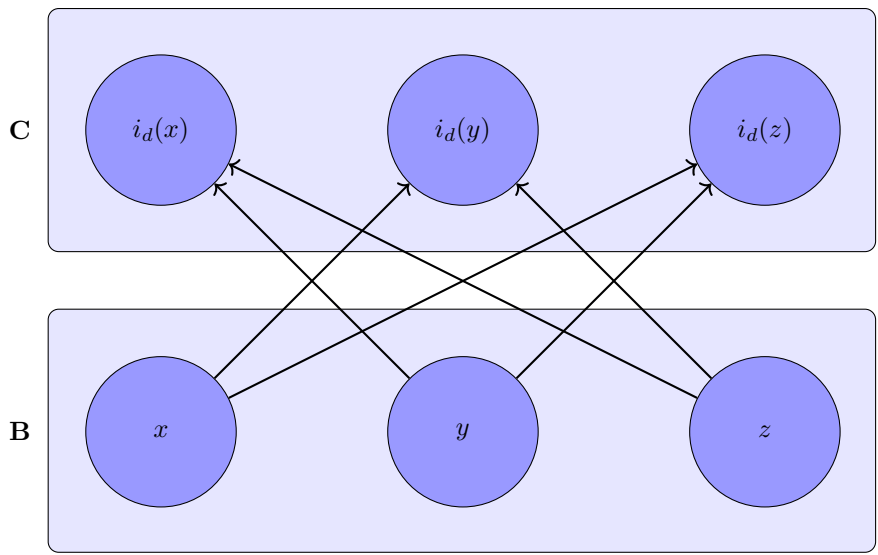
Notice, in particular, Proposition \ref{strict mono no R multi} shows there are preordered spaces with strict monotones and without strict monotone multi-utilities with cardinality $\mathfrak{c}$, which is not true for injective monotones (see \cite[Proposition 4]{hack2022representing}). It also shows that the class of preordered spaces where multi-utilities with cardinality $\mathfrak{c}$ exist is strictly contained inside the class of preordered spaces with multi-utilities, which consists of all preordered spaces \cite[Proposition 1]{evren2011multi}. In fact, as we show in Proposition \ref{no strict mono no R multi} through a variation of the preorder in Proposition \ref{R muli-uti no strict mono}, given an uncountable set $I$, there exist preordered spaces where neither strict monotones nor multi-utilities with the cardinality of $I$ exist. Notice the proof of Proposition \ref{no strict mono no R multi} relies on Szpilrajn extension theorem and, thus,
is non-constructive.

\begin{proposition}
\label{no strict mono no R multi}
If $I$ is an uncountable set, then there exist preordered spaces where neither multi-utilities with the cardinality of $I$ nor strict monotones exist.
\end{proposition}

\begin{proof}
Consider $X \coloneqq B \cup C$, where $B$ and $C$ are two copies of $\mathcal{P}(I)$ and consider on both $C$ and $B$ the total order $\preceq_S$ that results from applying Szpilrajn extension theorem \cite{szpilrajn1930extension} to the partial order defined by set inclusion on $\mathcal{P}(I)$. Furthermore, equip $X$ with $\preceq$ where
\begin{equation}
        x \preceq y \iff 
    \begin{cases}
    x \preceq_S y \text{ and } x,y \in B \\
    x \preceq_S y \text{ and } x,y \in C \\
    i_d(x) \preceq_S y, \text{ } x \in B \text{ and } y \in C \\
    x \prec_S i_d(y), \text{ } x \in C \text{ and }y \in B
    \end{cases}
    \end{equation}
$\forall x,y \in X$ with $i_d:B \to C$ the identity on $\mathcal{P}(I)$ (see Figure \ref{fig 7} for a representation of $\preceq$). Notice $(X,\preceq)$ is a preordered space. In analogy to Propositions \ref{R muli-uti no strict mono} and \ref{strict mono no R multi}, one can show that any family $(A_j)_{j \in J}$ of increasing subsets $A_j \subseteq X$ that separates $x$ and $y$ whenever  $x \prec y$ has larger cardinality than $I$. Since the existence of some $(A_j)_{j \in J}$ with those properties and $|J| \leq |I|$ is implied by both the existence of a multi-utility with the cardinality of $I$ by Lemma \ref{set charac} $(i)$ and the existence of a strict monotone by Lemma \ref{set charac} $(ii)$, we obtain that there is no multi-utility with the cardinality of $I$ nor a strict monotone for $X$.
\end{proof}

Proposition \ref{no strict mono no R multi} shows that, for any cardinal, there are preorders where both optimization principles do not exist and the dimension is larger than the cardinal. On the contrary, whenever the complexity is countable, (injective) optimization principles always exist \cite{hack2022representing}.
Note that the preorder we introduced in Proposition \ref{no strict mono no R multi} actually supports a stronger statement, which we include in Corollary \ref{stronger}. In order to prove it, we simply follow Proposition \ref{no strict mono no R multi} and add the fact that, as in Corollary \ref{R multi-ut no inj mono}, $(\chi_{i(x)})_{x \in X}$ is a multi-utility with the cardinality of
$\mathcal{P}(I)$.

\begin{corollary}
\label{stronger}
If $I$ is an uncountable set, then there exist preordered spaces where multi-utilities with the cardinality of $\mathcal{P}(I)$ exist, although neither strict monotones nor multi-utilities with the cardinality of $I$ do.
\end{corollary}

Regarding optimization and complexity, Corollary \ref{stronger} can be interpreted as Corollary Proposition \ref{no strict mono no R multi}.
To complement Propositions \ref{strict mono no R multi} and \ref{no strict mono no R multi}, we show, in Corollary \ref{|A| MU no R-P}, for any uncountable set $I$ there exist preorders which have multi-utilities with the cardinality of $I$ and no strict monotones. Notice, again, we follow the basic construction in Proposition \ref{R muli-uti no strict mono}, although we use the same non-constructive approach in Proposition \ref{no strict mono no R multi}.

\begin{figure}[!tb]
\centering
\begin{tikzpicture}
\node[rounded corners, draw,fill=blue!10, text height = 3cm, minimum width = 11cm,xshift=4cm,label={[anchor=west,left=.1cm]180:\textbf{C}}] {};
\node[rounded corners, draw,fill=blue!10, text height = 3cm, minimum width = 11cm,xshift=4cm,yshift=-4cm,label={[anchor=west,left=.1cm]180:\textbf{B}}] {};
    \node[main node] (1) {$i_d(x)$};
    \node[main node] (2) [right = 2cm  of 1]  {$i_d(y)$};
    \node[main node] (3) [right = 2cm  of 2]  {$i_d(z)$};
    \node[main node] (4) [below = 2cm  of 1] {$x$};
    \node[main node] (5) [right = 2cm  of 4] {$y$};
    \node[main node] (6) [right = 2cm  of 5] {$z$};

    \path[draw,thick,->]
    (4) edge node {} (1)
    (4) edge node {} (2)
    (4) edge node {} (3)
    (5) edge node {} (3)
    (6) edge node {} (3)
    (1) edge node {} (5)
    (2) edge node {} (6)
    (1) edge node {} (6)
    (5) edge node {} (2)
    (4) edge node {} (5)
    (5) edge node {} (6)
    (1) edge node {} (2)
    (2) edge node {} (3)
    (1) edge [bend left] node {} (3)
    (4) edge [bend right] node {} (6)
    ;
\end{tikzpicture}
\caption{Representation of a preordered space, defined in Proposition \ref{no strict mono no R multi}, which has no multi-utility with the cardinality of an uncountable set $I$ and no strict monotone. In particular, we show how $x,y,z \in B$, $x \prec y \prec z$, are related to $i_d(x) \prec i_d(y) \prec i_d(z) \in C$ where $B,C \coloneqq \mathcal{P}(I)$ and $i_d:B \to C$ is the identity on $\mathcal{P}(I)$. Notice an arrow from an element $w$ to an element $t$ represents $w \prec t$. Notice, also, this preorder is, essentially, the same as the one in the proof of Proposition \ref{R muli-uti no strict mono} (see Figure \ref{fig 4}) with a larger ground set. As a result, we used a non-constructive argument relying on Szpilrajn extension theorem to define it.}
\label{fig 7}
\end{figure}
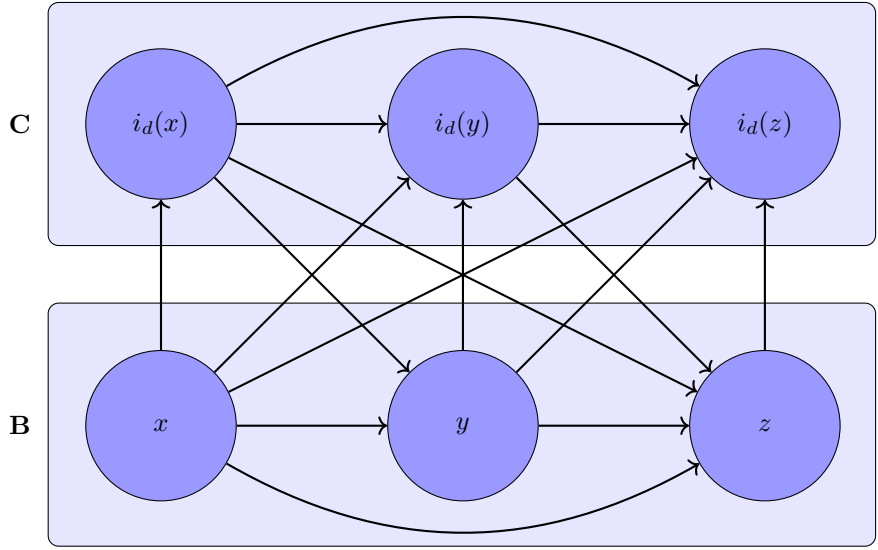

\begin{corollary}
\label{|A| MU no R-P}
If $I$ is an uncountable set, then there exist preordered spaces which have multi-utilities with the cardinality of $I$ and no strict monotone.
\end{corollary}

\begin{proof}
Consider $X \coloneqq C \cup B$, where $C$ and $B$ are two copies of $I$, and equip it with a preorder analogous to the one in Proposition \ref{no strict mono no R multi}. Notice $(\chi_{i(x)})_{x \in X}$ is a multi-utility for $X$. By slightly modifying the argument in Proposition \ref{no strict mono no R multi}, we conclude there are no strict monotones.
\end{proof}

Corollary \ref{|A| MU no R-P} shows that, unlike the countability ones \cite{alcantud2016richter,hack2022representing}, uncountability restrictions on the dimension of a preorder have no effect in general on the existence of optimization principles. Moreover,
as we show in Corollary \ref{last coro}, we can put together the preorders from Proposition \ref{mulit-ut and cardi} $(i)$ and Corollary \ref{|A| MU no R-P} to improve the relation between multi-utilities and the cardinality of $X/\mathord{\sim}$ even more.

\begin{corollary}
\label{last coro}
If $I$ is an uncountable set, then there exist preordered spaces where $|X/\mathord{\sim}|>|I|$ and multi-utilities with the cardinality of $I$ exist, while strict monotones do not.
\end{corollary}

\begin{proof}
Take $X \coloneqq A \cup B$ where $A$ and is the ground sets of the preorder in Corollary \ref{|A| MU no R-P} and $B$ is the ground set of the preorder in Proposition \ref{mulit-ut and cardi} $(i)$ without the empty set. We equip $X$ with the preorder in Corollary \ref{|A| MU no R-P} on $A$ and that of  Proposition \ref{mulit-ut and cardi} $(i)$ on $B$, leaving $x \bowtie y$ $\forall x,y \in X$ such that $x \in A$ and $y \in B$ or vice versa.
Since any strict monotone on $X$ would also be a strict monotone on $A$, they do not exist by Corollary \ref{|A| MU no R-P}. Notice we have $|X/\mathord{\sim}|>|I|$, since $|B/\mathord{\sim}|>|I|$. Notice, also,
$(g_i)_{i \in I} \cup (h_y)_{y \in A}$ is a multi-utility with the cardinality of $I$ for $X$, where $\forall i \in I$ $g_i(x) \coloneqq f_i(x)$ if $x \in B$ and $g_i(x) \coloneqq 0$ if $x \in A$, with $(f_i)_{i \in I}$ defined as in Proposition \ref{mulit-ut and cardi} $(i)$, and $\forall y \in A$ $h_y(x) \coloneqq \chi_{i(y)}$ if $x \in A$ and $h_y(x) \coloneqq 0$ if $x \in B$.
\end{proof}

Corollary \ref{last coro} also deals with the connections between optimization and complexity, showing that, even if the uncountability restrictions on the dimension do not apply to the decision space, there still exist preference spaces with no optimization principle.

To finish this section, since we have been mainly concerned with preordered spaces with infinite multi-utilities and uncountable $X/\mathord{\sim}$, we address both finite multi-utilities and countable $X/\mathord{\sim}$. The first thing to notice is the existence of finite multi-utilities does not imply $X/\mathord{\sim}$ is countable. This is exemplified by majorization \cite{marshall1979inequalities,arnold2018majorization}, since it is defined through a finite multi-utility \eqref{uncert rela} but the corresponding quotient space $\mathbb P_\Omega/\mathord{\sim_U}$ has the cardinality of the continuum.
It is straightforward to see, whenever $X/\mathord{\sim}$ is finite, there exists a finite multi-utility (see Corollary \ref{R multi-ut no inj mono}). However, as we show in Proposition \ref{no finite mu}, there exist preorders where $X/\mathord{\sim}$ is countably infinite and finite multi-utilities do not exist. Notice the preorder that supports this claim is, essentially, the same as the one in  Proposition \ref{strict mono no R multi}. However, in Proposition \ref{no finite mu}, we follow a simpler proof.

\begin{proposition}
\label{no finite mu}
There are preordered spaces $(X,\preceq)$ where $X/\mathord{\sim}$ is countably infinite and no finite multi-utilities exist.
\end{proposition}

\begin{proof}
Consider $ X \coloneqq \mathbb{Z}/\{0\}$ equipped with $\preceq$ where
\begin{equation*}
     n \preceq m \iff 
    \begin{cases}
    n=m\\
    n>0,\text{ } m <0 \text{ and } n \neq -m.
    \end{cases}
\end{equation*}
Notice $(\mathbb{Z}/\{0\})/\mathord{\sim}$ is countable. Assume there exists a finite multi-utility $(u_i)_{i=1}^k$. Notice for any pair $n,-n$ we have $n \bowtie -n$ and there must be some $i_n$ such that $u_{i_n}(-n) < u_{i_n}(n)$ by definition of multi-utility. If we consider, however, some $m \neq n$, then we have $u_{i_n}(m) \leq u_{i_n}(-n) < u_{i_n}(n) \leq u_{i_n}(-m)$. Thus, $i_m \neq i_n$. Considering w.l.o.g. $i_n=n$, we get $u_i(k+1) < u_i(-(k+1))$ for $i=1,..,k$. Thus, there is no multi-utility of cardinality $k$ for any $k<\infty$.
\end{proof}

Proposition \ref{no finite mu} shows that countability restrictions on the decision space do not necessarily imply finite bounds on the dimension of the preference space and, thus, on its complexity.
As a result of Proposition \ref{no finite mu}, there are preorders where countably infinite multi-utilities exist while finite ones do not. The preorder we used had, however, a countable $X/\mathord{\sim}$. We therefore complement this statement by showing in Proposition \ref{improve} that there are preorders with the same characteristics but uncountable $X/\mathord{\sim}$. 

\begin{proposition}
\label{improve}
There are preordered spaces $(X,\preceq)$ where $X/\mathord{\sim}$ is uncountable and, although countable multi-utilities exist, finite multi-utilities do not.
\end{proposition}

\begin{proof}
Let $\mathcal{P}_{inf}(\mathbb{N})$ be the set of infinite subsets of $\mathbb{N}$. Consider $X \coloneqq (\mathbb{N} \cup \mathcal{P}_{inf}(\mathbb{N}),\preceq)$ equipped with the preorder $\preceq$
\begin{equation*}
    x \preceq y \iff 
    \begin{cases}
    x=y\\
    x  \in \mathbb{N},\text{ } y \in \mathcal{P}_{inf}(\mathbb{N}) \text{ and } x \in y
    \end{cases}
\end{equation*}
$\forall x,y \in X$. Clearly, $|X/\mathord{\sim}|= \mathfrak{c}$, thus uncountable.

One can see $U \coloneqq (u_n,v_n)_{n \geq 0}$ is a countable multi-utility, where $u_n(x) \coloneqq 1$ if $x = n$ or $n \in x \in \mathcal{P}_{inf}(\mathbb{N})$ and $u_n(x) \coloneqq 0$ otherwise, and $v_n(x) \coloneqq 1$ if $n \not \in x$ and $x \in \mathcal{P}_{inf}(\mathbb{N})$ and $u_n(x) \coloneqq 0$ otherwise. Notice if $x \preceq y$ and $x \neq y$ then $x \in \mathbb{N}$ and $x \in y$. Thus, $u(x) \leq u(y)$ $\forall u \in U$. Assume now we have $\neg (x \preceq y)$. If $y \prec x$, then $y \in \mathbb{N}$ and $x \in \mathcal{P}_{inf}(\mathbb{N})$. Thus, there exists $m \in x$ such that $m \neq y$ and $u_m(y) < u_m(x)$. If $x \bowtie y$, then we consider four cases. If $x,y \in \mathbb{N}$, then $u_x(x)>u_x(y)$. If $x,y \in \mathcal{P}_{inf}(\mathbb{N})$, then, if there exists $n \in x/y$, we have $u_n(x)>u_n(y)$. Otherwise, there exists $n \in y/x$ and we have $v_n(x)>v_n(y)$. If $x \in \mathbb{N}$ and $y \in \mathcal{P}_{inf}(\mathbb{N})$, then $x \not \in y$ and we have $u_x(x)>U_x(y)$. If $y \in \mathbb{N}$ and $x \in \mathcal{P}_{inf}(\mathbb{N})$, then $y \not \in x$ and we have $v_y(x)>v_y(y)$. 

To conclude, we show there is no finite multi-utility. Let $A_0 \subseteq \mathcal{P}_{inf}(\mathbb{N})$, fix some $k \in \mathbb{N}$ and consider $(b_i)_{i=1}^{k+1} \subseteq A_0$, where $b_i \neq b_j$ if $i \neq j$, and $(A_i)_{i=1}^{k+1}$, where $A_i \coloneqq A_0/b_i$ for $i=1,..,k+1$.
Notice
$(b_i,A_i)_{i=1}^{k+1}$ is
a finite portion of the preorder in Proposition \ref{no finite mu}, since we have $b_i \preceq A_j$ if and only if $i \neq j$, and 
we can argue analogously as we did there
that no multi-utility with cardinality $k$ exists. 
Since $k$ is arbitrary, we obtain there is no finite multi-utility.  
\end{proof}

Proposition \ref{improve} proves we can relax the restriction on the decision space in Proposition \ref{no finite mu} from countable to uncountable and still find preference spaces with countably infinite dimension.


The preorder we introduced in Proposition \ref{improve} can, in fact, be used to improve the relation between real-valued monotones characterization of preorders and order density properties.
A subset $Z\subseteq X$, such that $x\prec y$ implies that there exists $z\in Z$ with $x\preceq z \preceq y$ is called \emph{order dense in the sense of Debreu} (or \emph{Debreu dense} for short) \cite{ok2002utility,bridges2013representations}. Accordingly, we say that $(X, \preceq)$ is \emph{Debreu separable} \cite{mehta1986existence} if there exists a countable Debreu dense set in $(X,\preceq)$.
Similarly, $(X,\preceq)$ is called
\emph{Debreu upper separable} if there exists a countable subset which is both Debreu dense and Debreu upper dense \cite{hack2022representing} (we defined Debreu upper dense subsets right before Proposition \ref{mulit-ut and cardi}).
As was shown in \cite[Proposition 9]{hack2022representing}, Debreu upper separable preorders have countable multi-utilities. However, there exist preorders which have countable multi-utilities but are not Debreu separable, like majorization for $|\Omega| \geq 3$ (see \cite[Lemma 5 (ii)]{hack2022representing}). In Proposition \ref{complem}, we complement these results 
by showing a preorder where countable multi-utilities exist and countable Debreu upper dense subsets do not. In particular, we show the preorder we introduced in Proposition \ref{improve} has no countable Debreu upper dense subsets although, as we showed there, it has countable multi-utilities. Notice, a preorder where the weaker fact that injective monotones exist and countable Debreu upper dense subsets do not can be found in \cite[Propsition 8]{hack2022representing}. There, an injective monotone was introduced and, although it was shown no countable multi-utility exists, it is easy to see any Debreu upper dense subset would be uncountable.

\begin{proposition}
\label{complem}
There are preordered spaces where countable multi-utilities exist and every Debreu upper dense subset is uncountable.
\end{proposition}

\begin{proof}
Consider the preorder $X \coloneqq (\mathbb{N} \cup \mathcal{P}_{inf}(\mathbb{N}),\preceq)$ from Proposition \ref{improve}. As we showed there, countable multi-utilities exist. Assume there exists a Debreu upper dense subset $D \subseteq X$. Consider $y \in \mathcal{P}_{inf}(\mathbb{N})$, $y \neq \mathbb{N}$. Notice there exists some $n_y \in \mathbb{N}/y$ and $y \cup \{n_y\} \bowtie y$. Since $D$ is Debreu upper dense, there exists some $d \in D$ such that $y \cup \{n_y\} \bowtie d \preceq y$. Since $d \preceq y$ implies either $d \in y$ or $d=y$, and $d \in y$ implies $d \in y \cup \{n_y\}$, thus $d \preceq y \cup \{n_y\}$ contradicting the definition of $d$, we have $d=y$. As a result, $\mathcal{P}_{inf}(\mathbb{N})/\{\mathbb{N}\} \subseteq D$ and $D$ is uncountable.
\end{proof}

Although usually expressed as an equivalence, for total preorders, between the existence of utility functions and that of countable Debreu dense subsets, the classical result by Debreu (see \cite[Theorem 1.4.8]{bridges2013representations}) can, alternatively, be stated as the equivalence between the existence of countable multi-utilities and that of countable Debreu separable subsets. From this perspective, the interest of the result lies in the fact that the existence of a countable family of increasing sets that separate (in the sense of Lemma \ref{set charac}) the elements in a preorder results in a countable subset of elements with this separation (in the sense of order density) property. Since the natural extension of Debreu separability to non-total preorders is Debreu upper separability, Proposition \ref{complem}, together with \cite[Propsition 8]{hack2022representing}, shows that the transition from separating sets to elements does not hold in general for non-total preorders.

Notice, although they coincide when they are countable (see \cite[Proposition 4.1]{alcantud2016richter} and \cite[Proposition 6]{hack2022representing}), it remains an open question how the different sorts of multi-utilities relate to each other when they are finite. As a first result in this direction, we finish with a characterization of preordered spaces with finite injective monotones multi-utilities. 

\begin{proposition}
\label{finite relations}
If $(X,\preceq)$ is a preordered space, then the following are equivalent:
\begin{enumerate}[label=(\roman*)]
\item There exists a finite multi-utility $(u_i)_{i \leq N}$ such that the image of the non-injective set
\begin{equation}
\label{non-inj}
    I_{u_i} \coloneqq \{r \in \mathbb{R}| \exists x,y \in X \text{ such that } x,y \in u_i^{-1}(r) \text{ and } \neg(x \sim y)\}
\end{equation}
is countable $\forall i \leq N$.
\item There exists a finite injective monotone multi-utility $(v_i)_{i \leq N}$.
\end{enumerate}
\end{proposition}

\begin{proof}
By definition, given an injective monotone multi-utility $(v_i)_{i \leq N}$, we have $I_{v_i}= \emptyset$ $\forall i \leq N$. Conversely, consider $u \in (u_i)_{i \leq N}$ a monotone such that the image of its non-injective set $I_u$ is countable. Take $(r_n)_{n \geq 0}$ a numeration of $I_u$, $(y_n)_{n \geq 0} \subseteq X$ a set such that $u(y_n)=r_n$ $\forall n \geq 0$ and, w.l.o.g., an injective monotone $c_0: X \to (0,1)$. Notice injective monotones exist under the hypotheses, as we showed in \cite[Proposition 5]{hack2022representing}. 
Define, then, 
\begin{equation*}
    w_0(x):= 
    \begin{cases}
    u(x) & \text{if} \text{ } u(x) < r_0\\
    u(x) + c_0(x) & \text{if} \text{ } u(x) = r_0\\
    u(x)+1 & \text{else.}
    \end{cases}
\end{equation*}
$\forall x \in X$. Notice $I_{w_0} \subset I_u$, since $x_0 \not \in I_{w_0}$, and we have both $u(x) \leq u(y)$ implies $w_0(x) \leq w_0(y)$ and $u(x) < u(y)$ implies $w_0(x) < w_0(y)$ $\forall x,y \in X$.
Similarly, consider a family of injective monotones $(c_n)_{n \geq 1}$ such that $c_n:X \to (0,2^{-n})$ for $n \geq 1$
and define, also for $n \geq 1$,
\begin{equation*}
    w_{n}(x):= 
    \begin{cases}
    w_{n-1}(x) & \text{if} \text{ } w_{n-1}(x) < w_{n-1}(y_n)\\
    w_{n-1}(x) + c_n(x) & \text{if} \text{ } w_{n-1}(x) = w_{n-1}(y_n)\\
    w_{n-1}(x)+2^{-n} & \text{else}
    \end{cases}
\end{equation*}
$\forall x \in X$.
Notice $I_{w_{n-1}} \subset I_{w_n}$ holds $\forall n \geq 1$, since $x_n \not \in I_{w_n}$, and we have both $w_{n-1}(x) \leq w_{n-1}(y)$ implies $w_n(x) \leq w_n(y)$ and $w_{n-1}(x) < w_{n-1}(y)$ implies $w_n(x) < w_n(y)$ $\forall x,y \in X$. Lastly, consider the pointwise limit $v(x) \coloneqq \lim_{n \to \infty} w_n(x)$. Notice $v$ is well-defined and, also, 
an injective monotone, since $I_v= \emptyset$ by construction.

Following the same procedure for each monotone in $(u_i)_{i \leq N}$, we get a family of injective monotones $(v_i)_{i \leq N}$. To conclude it is a multi-utility, we need to show, $\forall x,y \in X$ with $\neg(x \preceq y)$, there exists some $i \leq N$ such that $v_i(x)>v_i(y)$. If $y \prec x$, then $v_i(x)>v_i(y)$ $\forall i \leq N$ by definition of injective monotone. Otherwise, if $x \bowtie y$, there exists some $i \leq N$ such that $u_i(x) > u_i(y)$. Thus, we also have $v_i(x) > v_i(y)$. Hence, $(v_i)_{i \leq N}$ is a multi-utility.
\end{proof}

Notice we can weaken the hypothesis, assuming, instead of \eqref{non-inj}, that
\begin{equation*}
 \{r \in \mathbb{R}| \exists x,y \in X \text{ such that } x,y \in u_i^{-1}(r) \text{ and } x \prec y\}
\end{equation*}
is countable $\forall i \leq N$, to conclude, analogously, that the existence of finite multi-utilities and that of finite strict monotone multi-utilities are equivalent. Notice, as a result, we obtain the existence of finite multi-utilities coincides with that of finite strict monotone multi-utilities and that of finite injective monotone multi-utilities whenever  $X/\mathord{\sim}$ is countable. The general case where $X/\mathord{\sim}$ is uncountable (in particular, when $|X/\mathord{\sim}|\leq \mathfrak{c}$ since, otherwise, there are no injective monotones), remains open. This is due to the fact the technique in Proposition \ref{finite relations} cannot be used and $I_{u_i}$ is not necessarily countable $\forall i \leq N$, as one can see in majorization, for example. There, taking $u_i$ as in \eqref{uncert rela}, we have $(\frac{i}{|\Omega|},1) \subseteq I_{u_i}$ $\forall i \leq |\Omega|-1$ and, thus, $I_{u_i}$ is uncountable $\forall i \leq |\Omega|-1$. Notice, also, the technique in Proposition \ref{finite relations} is similar to the one we used in \cite[Proposition 2]{hack2022representing}, where we showed the existence of an injective monotone is equivalent to that of a strict monotone $f$ whose non-injective set
\begin{equation*}
\{ x \in X| \text{ }\exists y \in X \text{ }s.t. \text{ } f(x)=f(y) \text{ and } x \bowtie y \}
\end{equation*}
is countable. Notice the hypothesis there is stronger, since the hypothesis that the image of the non-injective set $I_f$ is countable is insufficient, as one can see using the preorder in \cite[Proposition 1 (i)]{hack2022representing}.

The technique in Proposition \ref{finite relations} can actually be used to prove that countable multi-utilities and countable injective monotone multi-utilities always coincide (see \cite[Proposition 6]{hack2022representing}). The only detail of importance is, whenever a countable multi-utility exists, there exists, by Lemma \ref{set charac} $(i)$, a countable family of increasing sets $(A_n)_{n \geq 0}$ that $\forall x,y\in X$ with $x \prec y$ separates $x$ from $y$ and $\forall x,y\in X$ with $x \bowtie y$ separates both $x$ from $y$ and $y$ from $x$. In particular, $(\chi_{A_n})_{n \geq 0}$ is a countable multi-utility with the property that
$I_{\chi_{A_n}}$ is finite $\forall n \geq 0$. Since injective monotones exist, we can follow Proposition \ref{finite relations} to construct a countable injective monotone multi-utility.

\section{Discussion}
\label{discussion}

In this work, we have improved the classification of preordered spaces through real-valued monotones in terms of the cardinality of multi-utilities and quotient spaces, c.f. Figure \ref{fig:classification}, and,
as a result, we have contributed to the study of complexity, optimization and their relation in preordered spaces. 


\paragraph{Classification of preordered spaces through real-valued monotones.}
The state of the classification of preordered spaces in terms of real-valued monotones can be found in Figure \ref{fig:classification}, whereas our contributions are shown in Figure \ref{fig:contributions}. In this paragraph, we summarize the relation between the different classes and distinguish between our results and the ones in the literature.
We will begin from the innermost class, preorders with utility functions, and finish with the outermost class, which contains all preorder \cite{evren2011multi}, that is, preorders with multi-utilities.

The relation between utility functions and the subsequent classes, finite multi-utilities and preorders with countable $X/\mathord{\sim}$ is as follows. A utility function is a finite multi-utility, although there are preordered spaces where finite multi-utilities exist and utilities do not, like majorization \cite{arnold2018majorization,marshall1979inequalities}. We can also use majorization to show there are preorders with a finite multi-utility where $X/\mathord{\sim}$ is uncountable. By Proposition \ref{no finite mu}, a countable $X/\mathord{\sim}$ does not imply there exists a finite multi-utility. Notice, also, preorders with utilities can have an uncountable $X/\mathord{\sim}$, the easiest example being $(\mathbb R,\leq)$, and any non-total preorder with countable $X/\mathord{\sim}$ has no utility function.

The next class of interest are preorders with countable multi-utilities, which are exactly those with countable strict monotone multi-utilities \cite[Proposition 4.1]{alcantud2016richter} and countable injective monotone multi-utilities \cite[Proposition 6]{hack2022representing}. By Proposition \ref{improve}, there are preorders with countable multi-utilities where $X/\mathord{\sim}$ is uncountable such that no finite multi-utility exists, although finite multi-utilities are, of course, countable.
Also, whenever $X/\mathord{\sim}$ is countable, there exists a countable multi-utility, namely, $(\chi_{i(x)})_{[x] \in X/\mathord{\sim}}$ \cite{evren2011multi}.

The following wider category are preorders with injective monotones, which are equivalent to those with injective monotone multi-utilities with cardinality $\mathfrak{c}$ by \cite[Proposition 4]{hack2022representing}. As we showed in \cite[Proposition 5]{hack2022representing}, injective monotones can be constructed from countable multi-utilities. However, again by \cite[Proposition 8]{hack2022representing}, the converse is false.

Injective monotones are contained inside two classes: preorders with strict monotone multi-utilities of cardinality $\mathfrak{c}$ and preorders where $|X/\mathord{\sim}|\leq\mathfrak{c}$. It is straightforward to see $|X/\mathord{\sim}|\leq\mathfrak{c}$ whenever injective monotones exist. Because of this, since it implies multi-utilities of cardinality $\mathfrak{c}$ exist \cite{evren2011multi}, and Proposition \ref{R-P multi charac}, strict monotone multi-utilities of cardinality $\mathfrak{c}$ exist whenever injective monotones do. However, by Proposition \ref{R-P multi charac} and Corollary \ref{R multi-ut no inj mono}, there are preordered spaces with strict monotone multi-utilities of cardinality $\mathfrak{c}$ and without injective monotones. Similarly, as Proposition \ref{strict mono no injective} shows, there are preorders where we have $|X/\mathord{\sim}| \leq \mathfrak{c}$ and no injective monotones. Moreover, by Proposition \ref{mulit-ut and cardi} (ii), having strict monotone multi-utilities of cardinality $\mathfrak{c}$ does not imply $|X/\mathord{\sim}| \leq \mathfrak{c}$. Conversely, as noticed in Corollary \ref{cardin no strict mono}, we also get a negative result if we interchange the role of both clauses, that is, there are preorders where $|X/\mathord{\sim}| \leq \mathfrak{c}$ holds and no strict monotone multi-utility of cardinality $\mathfrak{c}$ exists. Notice the preorder in Corollary \ref{cardin no strict mono} was, essentially, already introduced by Debreu in \cite{debreu1954representation}. As we stated in Proposition \ref{R-P multi charac}, having a strict monotone and a multi-utility of cardinality $\mathfrak{c}$, the following class of interest, is equivalent to having a strict monotone multi-utility of that cardinality.
However, by Proposition \ref{R muli-uti no strict mono}, $|X/\mathord{\sim}| \leq \mathfrak{c}$ does not imply there exists a strict monotone multi-utility of cardinality $\mathfrak{c}$. If we relax the implication of the statement to \emph{multi-utility of cardinality $\mathfrak{c}$}, then it is indeed true (see \cite{evren2011multi} or Corollary \ref{R multi-ut no inj mono}). 
There are, actually, preorders with a multi-utility of cardinality $\mathfrak{c}$ and no strict monotone multi-utility of that cardinality such that $|X/\mathord{\sim}|>\mathfrak{c}$, as Corollary \ref{last coro} shows. 
Finally, by Proposition \ref{strict mono no R multi}, there are preorders where strict monotones exist a multi-utilities of cardinality $\mathfrak{c}$ do not. In fact, by Proposition \ref{no strict mono no R multi}, there are preorders without both strict monotones and multi-utilities of cardinality $\mathfrak{c}$. This completes the results which are needed to construct Figure \ref{fig:classification}. Notice, although we have focused on the case $I=\mathbb{R}$, many of the results hold for a general uncountable set $I$, as we stated them in Section \ref{sect: class}. 

Aside from those in the last paragraph, there are four more results in Section \ref{sect: class}. Proposition \ref{finite relations} shows the equivalence between finite multi-utilities and finite injective monotone multi-utilities in well-behaved cases. Notice the only finite case which appears in Figure \ref{fig:classification} is that of multi-utilities, as the relation with the other types remains to be clarified. Proposition \ref{complem} improves upon \cite{hack2022representing}, where it was shown Debreu upper separable preorders have countable multi-utilities \cite[Proposition 9]{hack2022representing} while there are preorders with countable multi-utilities which are not Debreu separable \cite[Lemma 5]{hack2022representing}, by showing there exist preorders with countable multi-utilities where every Debreu upper dense subset is uncountable. Lastly, Corollary \ref{stronger} is slightly stronger than Proposition \ref{no strict mono no R multi} and uses the same preorder, while  Corollary \ref{|A| MU no R-P} is weaker than Corollary \ref{last coro}.  

\paragraph{Complexity and optimization.} Since the minimal cardinality of the existing multi-utilities can be used as a measure of complexity and the existence of optimization principles can be reformulated in terms of strict and injective monotones \cite{hack2022representing}, the classification of preorders in terms of real-valued monotones improves our knowledge regarding complexity, optimization and the connections between them. Although we omit it here for the sake of brevity, we can interpret the classification of preorders according to monotones (cf. the paragraph above and Figure \ref{fig:classification}) in terms of complexity and optimization (as we did right after presenting each result in Section \ref{sect: class}).

\paragraph{Debreu dimension.} There is a notion of dimension for partial orders which goes back to \cite{dushnik1941partially}, it has remained somewhat disconnected from the more intuitive geometrical notion, which corresponds to multi-utilities. In fact, there exist preorders where the classical definition of dimension is finite while the geometrical one is uncountable. In \cite{hack2022geometrical}, we propose a variation of the classical notion, called Debreu dimension, and, using results from this work, show that such a disconnection between this definition and the geometrical one does not occur. That is, we show that the geometrical dimension is countable if and only if the Debreu dimension also is. 

\paragraph{Open questions.}
Several scientific disciplines rely on preordered spaces and their representation via real-valued monotones. Thus, refining the classification via the introduction of new classes and establishing more connections between separated classes in cases of interest would, potentially, improve several areas, like utility theory \cite{debreu1954representation,rebille2019continuous} and the study of social welfare relation \cite{banerjee2010multi} in economics, statistical estimation \cite{hennig2007some} is statistics, equilibrium thermodynamics \cite{lieb1999physics,candeal2001utility}, entanglement theory \cite{nielsen1999conditions,turgut2007catalytic} and general relativity \cite{bombelli1987space,minguzzi2010time} in physics and, lastly, multicriteria optimization \cite{jahn2009vector,ehrgott2005multicriteria}. Specific questions that remain to be solved include, for example, the relation between the different sorts of finite multi-utilities we have introduced. In particular, it is unclear whether Proposition \ref{finite relations} can be improved or preorders with finite multi-utilities and no finite injective multi-utilities exist.



\newpage 





\newpage

\bibliographystyle{plain}
\bibliography{main.bib}

\begin{thebibliography}{10}

\bibitem{alcantud2013representations}
Jos{\'e} Carlos~R Alcantud, Gianni Bosi, and Magal{\`\i} Zuanon.
\newblock Representations of preorders by strong multi-objective functions.
\newblock Technical Report MPRA Paper 5232, University Library of Munich, 2013.

\bibitem{alcantud2016richter}
Jos{\'e} Carlos~R Alcantud, Gianni Bosi, and Magal{\`\i} Zuanon.
\newblock Richter-{P}eleg multi-utility representations of preorders.
\newblock {\em Theory and Decision}, 80(3):443--450, 2016.

\bibitem{arnold2018majorization}
Barry~C Arnold.
\newblock {\em Majorization and the Lorenz order with applications in applied
  mathematics and economics}.
\newblock Springer, 2018.

\bibitem{banerjee2010multi}
Kuntal Banerjee and Ram~Sewak Dubey.
\newblock On multi-utility representation of equitable intergenerational
  preferences.
\newblock In {\em Econophysics and Economics of Games, Social Choices and
  Quantitative Techniques}, pages 175--180. Springer, 2010.

\bibitem{bombelli1987space}
Luca Bombelli, Joohan Lee, David Meyer, and Rafael~D Sorkin.
\newblock Space-time as a causal set.
\newblock {\em Physical review letters}, 59(5):521, 1987.

\bibitem{bosi2018upper}
G~Bosi, P~Bevilacqua, and M~Zuanon.
\newblock Upper semicontinuous representability of maximal elements for non
  total preorders on compact spaces.
\newblock {\em Res J Econ 2: 1}, 3:2, 2018.

\bibitem{bosi2020topologies}
Gianni Bosi, Asier Estevan, and Armajac Ravent{\'o}s-Pujol.
\newblock Topologies for semicontinuous {R}ichter-{P}eleg multi-utilities.
\newblock {\em Theory and Decision}, 88(3):457--470, 2020.

\bibitem{bosi2012continuous}
Gianni Bosi and Gerhard Herden.
\newblock Continuous multi-utility representations of preorders.
\newblock {\em Journal of Mathematical Economics}, 48(4):212--218, 2012.

\bibitem{bosi2016continuous}
Gianni Bosi and Gerhard Herden.
\newblock On continuous multi-utility representations of semi-closed and closed
  preorders.
\newblock {\em Mathematical Social Sciences}, 79:20--29, 2016.

\bibitem{bosi2013existence}
Gianni Bosi and M~Zuanon.
\newblock Existence of maximal elements of semicontinuous preorders.
\newblock {\em Int. J. Math. Anal}, 7:1005--1010, 2013.

\bibitem{brandao2015second}
Fernando Brandao, Micha{\l} Horodecki, Nelly Ng, Jonathan Oppenheim, and
  Stephanie Wehner.
\newblock The second laws of quantum thermodynamics.
\newblock {\em Proceedings of the National Academy of Sciences},
  112(11):3275--3279, 2015.

\bibitem{bridges2013representations}
Douglas~S Bridges and Ghanshyam~B Mehta.
\newblock {\em Representations of preferences orderings}, volume 422.
\newblock Springer Science \& Business Media, 2013.

\bibitem{candeal2001utility}
Juan~C Candeal, Juan~R De~Miguel, Esteban Indur{\'a}in, and Ghanshyam~B Mehta.
\newblock Utility and entropy.
\newblock {\em Economic Theory}, 17(1):233--238, 2001.

\bibitem{debreu1954representation}
Gerard Debreu.
\newblock Representation of a preference ordering by a numerical function.
\newblock {\em Decision processes}, 3:159--165, 1954.

\bibitem{debreu1959theory}
Gerard Debreu.
\newblock {\em Theory of value: An axiomatic analysis of economic equilibrium},
  volume~17.
\newblock Yale University Press, 1959.

\bibitem{debreu1964continuity}
Gerard Debreu.
\newblock Continuity properties of paretian utility.
\newblock {\em International Economic Review}, 5(3):285--293, 1964.

\bibitem{dushnik1941partially}
Ben Dushnik and Edwin~W Miller.
\newblock Partially ordered sets.
\newblock {\em American journal of mathematics}, 63(3):600--610, 1941.

\bibitem{ehrgott2005multicriteria}
Matthias Ehrgott.
\newblock {\em Multicriteria optimization}, volume 491.
\newblock Springer Science \& Business Media, 2005.

\bibitem{evren2011multi}
{\"O}zg{\"u}r Evren and Efe~A Ok.
\newblock On the multi-utility representation of preference relations.
\newblock {\em Journal of Mathematical Economics}, 47(4-5):554--563, 2011.

\bibitem{giles2016mathematical}
Robin Giles.
\newblock {\em Mathematical foundations of thermodynamics: International series
  of monographs on pure and applied mathematics}, volume~53.
\newblock Elsevier, 2016.

\bibitem{hack2022geometrical}
Pedro Hack, Daniel~A Braun, and Sebastian Gottwald.
\newblock On a geometrical notion of dimension for partially ordered sets.
\newblock {\em arXiv preprint arXiv:2203.16272}, 2022.

\bibitem{hack2022representing}
Pedro Hack, Daniel~A Braun, and Sebastian Gottwald.
\newblock Representing preorders with injective monotones.
\newblock {\em Theory and Decision}, pages 1--28, 2022.

\bibitem{hardy1952inequalities}
Godfrey~Harold Hardy, John~Edensor Littlewood, and George P{\'o}lya.
\newblock Inequalities cambridge university press.
\newblock {\em Cambridge, England}, page~89, 1952.

\bibitem{harzheim2006ordered}
Egbert Harzheim.
\newblock {\em Ordered sets}, volume~7.
\newblock Springer Science \& Business Media, 2006.

\bibitem{hennig2007some}
Christian Hennig and Mahmut Kutlukaya.
\newblock Some thoughts about the design of loss functions.
\newblock {\em REVSTAT--Statistical Journal}, 5(1):19--39, 2007.

\bibitem{herden1989existence}
Gerhard Herden.
\newblock On the existence of utility functions.
\newblock {\em Mathematical Social Sciences}, 17(3):297--313, 1989.

\bibitem{jahn2009vector}
Johannes Jahn.
\newblock {\em Vector optimization}.
\newblock Springer, 2009.

\bibitem{jaynes1957information}
Edwin~T Jaynes.
\newblock Information theory and statistical mechanics.
\newblock {\em Physical review}, 106(4):620, 1957.

\bibitem{jaynes2003probability}
Edwin~T Jaynes.
\newblock {\em Probability theory: The logic of science}.
\newblock Cambridge university press, 2003.

\bibitem{lieb1999physics}
Elliott~H Lieb and Jakob Yngvason.
\newblock The physics and mathematics of the second law of thermodynamics.
\newblock {\em Physics Reports}, 310(1):1--96, 1999.

\bibitem{marshall1979inequalities}
Albert~W Marshall, Ingram Olkin, and Barry~C Arnold.
\newblock {\em Inequalities: theory of majorization and its applications},
  volume 143.
\newblock Springer, 1979.

\bibitem{mehta1986existence}
Ghanshyam Mehta.
\newblock Existence of an order-preserving function on normally preordered
  spaces.
\newblock {\em Bulletin of the Australian Mathematical Society},
  34(1):141--147, 1986.

\bibitem{minguzzi2010time}
Ettore Minguzzi.
\newblock Time functions as utilities.
\newblock {\em Communications in Mathematical Physics}, 298(3):855--868, 2010.

\bibitem{nielsen1999conditions}
Michael~A Nielsen.
\newblock Conditions for a class of entanglement transformations.
\newblock {\em Physical Review Letters}, 83(2):436, 1999.

\bibitem{ok2002utility}
Efe~A Ok et~al.
\newblock Utility representation of an incomplete preference relation.
\newblock {\em Journal of Economic Theory}, 104(2):429--449, 2002.

\bibitem{peleg1970utility}
Bezalel Peleg.
\newblock Utility functions for partially ordered topological spaces.
\newblock {\em Econometrica: Journal of the Econometric Society}, pages 93--96,
  1970.

\bibitem{rebille2019continuous}
Yann R{\'e}bill{\'e}.
\newblock Continuous utility on connected separable topological spaces.
\newblock {\em Economic Theory Bulletin}, 7(1):147--153, 2019.

\bibitem{richter1966revealed}
Marcel~K Richter.
\newblock Revealed preference theory.
\newblock {\em Econometrica: Journal of the Econometric Society}, pages
  635--645, 1966.

\bibitem{szpilrajn1930extension}
Edward Szpilrajn.
\newblock Sur l'extension de l'ordre partiel.
\newblock {\em Fundamenta mathematicae}, 1(16):386--389, 1930.

\bibitem{turgut2007catalytic}
Sadi Turgut.
\newblock Catalytic transformations for bipartite pure states.
\newblock {\em Journal of Physics A: Mathematical and Theoretical},
  40(40):12185, 2007.

\bibitem{white1980notes}
DJ~White.
\newblock Notes in decision theory: Optimality and efficiency ii.
\newblock {\em European Journal of Operational Research}, 1980.

\end{thebibliography}

\end{document}